\documentclass[10pt,a4paper,oneside]{amsart}
\RequirePackage[l2tabu, orthodox]{nag}

\usepackage[utf8]{inputenc}
\usepackage[english]{babel}
\usepackage[T1]{fontenc}
\usepackage{microtype, fancyhdr, lmodern, xcolor, verbatim}
\usepackage[a4paper, hmarginratio=1:1]{geometry}
\usepackage[all]{xy}

\usepackage{amsfonts, amsmath, amsthm, amssymb, mathrsfs, amscd}
\usepackage{euscript, mathtools}
\numberwithin{equation}{section}	
\usepackage{faktor}
\allowdisplaybreaks

\usepackage{url, enumitem}
\usepackage[noadjust]{cite}
\usepackage{comment}

\usepackage{caption}
\usepackage{pinlabel, graphicx}
\usepackage{tikz}
\usepackage{upgreek}

\usepackage{hyperref}

\theoremstyle{plain}
\newtheorem{Theorem}{Theorem}[section]
\newtheorem{Proposition}[Theorem]{Proposition}

\newtheorem{Lemma}[Theorem]{Lemma}

\theoremstyle{remark}
\newtheorem{Remark}[Theorem]{Remark}
\newtheorem{OpenP}[Theorem]{Open Problem}

\theoremstyle{definition}
\newtheorem{Definition}[Theorem]{Definition}

\newtheorem{para}[Theorem]{$\circ$}

\newcommand{\redx}[1]{#1}
\newcommand{\redxx}[1]{}
\newcommand \bluexx [1]{}
\newcommand \yellowxx [1]{}

\newcommand\N{\mathbb{N}}
\newcommand\Z{\mathbb{Z}}

\newcommand\R{\mathbb{R}}

\newcommand\ii{i} 
\newcommand\jj{j} 
\newcommand\nn{n} 
\newcommand{\nnn}{m}

\newcommand{\catfont}{\mathsf} 

\newcommand \agg  {$\pi$-module}

\newcommand \anagg  {a $\pi$-module}
\newcommand \Anagg  {A $\pi$-module}
\newcommand{\aagg}{{\catfont{Agg}}} 
\newcommand{\aaagg}{\agg} 
\DeclareMathOperator{\Aut}{Aut}	
\newcommand\BB[1]{B_{#1}}	

\newcommand \C {{L_n}}  
\def \d {\partial} 
\def \E {\mathcal{E}}
\newcommand \F {\mathcal{F}}
\newcommand{\UU}{{\mathcal U}}      
\newcommand{\TOP}{ \mathrm{TOP}}   
\newcommand{\FFa}{{\mathcal F}^a}   
\newcommand{\Grp}{{\mathsf{Grp}}}

\newcommand \G {\mathcal{G}}
\def \H {\mathcal{H}}
\newcommand\id{\mathrm{id}} 
\newcommand\inv{^{-1}} 
\newcommand \J {\mathcal{J}}
\newcommand\LB[1]{LB_{#1}}	
\newcommand\LBE[1]{LB_{#1}^{\mathrm{ext}}}	

\newcommand \M {{\mathcal{M}}}
\DeclareMathOperator{\MCG}{MCG} 

\newcommand \trr {\triangleright}

\newcommand{\pn}{\theta} 
\newcommand{\Pn}{\Theta} 

\newcommand{\ignore}[1]{}
\newcommand{\beq}{\begin{equation}}
\newcommand{\eq}{\end{equation}}
    
\newcommand{\ssigma}{\boldsymbol{\Sigma}} 

\newcommand{\TAR}{lifted Artin representation}

\newcommand{\Rho}{\uprho} 
\newcommand{\Tau}{\uptau}

\newcommand{\SET}{{\mathcal{ SET}}} 
\newcommand{\FINSET}{{\mathcal{ FINSET}}} 
\newcommand{\Sym}{Sym}
\newcommand{\Cg}{Coxeter geometric}

\newcommand{\K}{\Upsilon}

\begin{document}

\title[On a canonical lift of Artin's representation]{On a canonical lift of 
Artin's representation to~loop~braid
  groups}
\date{\today}

\author[Damiani]{Celeste Damiani}
\email[Celeste Damiani]{c.damiani@leeds.ac.uk}

\author[Faria Martins]{Jo\~ao Faria Martins}
\email[Jo\~ao Faria Martins]{j.fariamartins@leeds.ac.uk}

\author[Martin]{Paul Purdon Martin}
\email[Paul Martin]{p.p.martin@leeds.ac.uk}
\address[Damiani, Faria Martins, Martin]{{School of Mathematics, University of 
Leeds, Leeds, LS2 9JT, United Kingdom}}

\thanks{This work was funded by the Leverhulme trust grant ``RPG-2018-029: 
Emergent Physics From Lattice Models of Higher Gauge Theory''. We would like to 
thank Alex Bullivant, Markus Szymik and Paolo Bellingeri for useful 
discussions.}


\begin{abstract}

Each pointed
topological space has
an associated \emph{\agg},
obtained from action of its first homotopy group 
on its second homotopy group.
For the $3$-ball with  
a trivial link with $n$-components removed from its interior,
its \agg\ $\M_n$ is of \emph{free type}.
In this paper we give an injection of the \emph{(extended) loop braid group} 
into the 
group of automorphisms of~$\M_n$.
We give a topological interpretation of this injection, showing that it is 
both an extension of Artin's representation for braid groups and of Dahm's 
homomorphism for (extended) loop braid groups.
\end{abstract}

\maketitle

\section{Introduction}
A recent paper~\cite{BCKMM2} 
showed that \emph{discrete higher gauge 
theory} (as discussed generally and informally by many authors, see 
e.g.~\cite[\S10.2]{Martin:1991}, and compare with \emph{continuous} higher gauge 
theory~\cite{Pfeiffer, baez_schreiber, BaezHuerta11,
martins_picken,schreiber_waldorf1} and several others) is well-defined.
The technical engine of the construction notably reflects the
\emph{combinatorial  homotopy}
  of Whitehead, 
  Baues, 
  \textit{et al.}~\cite{Baues4D}.
  Relatively simple aspects of the construction
such as loop-particle braiding~\cite{loopy}
  yield higher generalisations of 
  classical results, for example in low-dimensional topology
  \cite{martins_2009}. 
We discuss one such result, a lifting of Artin's representation of the 
\emph{braid group}~\cite{Birman:book} to the \emph{(extended) loop braid 
group}~\cite{Lin:2008}~(see \cite{Damiani}, for a survey).

\subsection{From mapping class groups to Artin-like representations}
\label{Int_1_1}
Let $X$ be an oriented topological manifold
 with boundary~$\partial X$,
and $A$ a
possibly empty subset in the interior of~$X$.
A \emph{self-homeomorphism of the pair $(X,A)$
relative to the boundary,}
is a self-homeomorphism $g$ of $X$ that fixes
$\partial X$ pointwise, $A$ setwise (i.e.~$g(A)=A$),
and preserves the orientation of~$X$. 
Two such homeomorphisms $f_0$ and $f_1$ are \emph{$(X,A)$-isotopic} 
if they can be included in a $1$-parameter family
$\{f_t\}_{t\in[0,1]}$ of self-homeomorphisms of $(X, A)$ relative to the boundary,
such that the map $X \times [0, 1] \to X$ sending $(x, t)$ to $f_t(x)$ is continuous. 
The mapping class group $\MCG(X,A)$ is the group of $(X,A)$-isotopy classes of self-homeomorphisms of~$(X,A)$.
We write $[g]$ for the class of~$g$. Our convention for the product in $\MCG(X,A)$ is:~$[g][g']=[g \circ g']$.

A self-homeomorphism $g \colon (X,A) \to (X,A)$ relative to the boundary,
takes an $\nn$-path
$[0, 1]^\nn \xrightarrow{\gamma} X$ 
to an $\nn$-path~$[0, 1]^\nn \xrightarrow{g \circ \gamma} X$.
In particular $g$ takes a loop based at a point $\ast$ in $\partial X$
to another such loop.
Furthermore
a homeomorphism that fixes a set fixes its complement.
If~$\gamma$ avoids $A$ then so does~$g\circ\gamma$.
The map~$g$ thus induces an automorphism
of~$\pi_1(X \setminus A,\ast) $
with~$\ast \in \partial X$, 
that is:~$ [\gamma] \mapsto  [g\circ\gamma]$.

Passing from $g$ to~$[g]$, 
this gives a well-defined group homomorphism
\begin{equation}
\label{eq:MtoAut}
\tau \colon \MCG(X,A) \longrightarrow \Aut(\pi_1(X \setminus A,\ast)),
\end{equation}
where~$ [g \circ \gamma] =  [g' \circ \gamma'] $,
if~$g'\in [g]$ and~$\gamma' \in [\gamma]$.
If $G$ is a group, our convention for the product in~$\Aut(G)$ is~$fg=f \circ g$.

By considering based $n$-paths {$([0,1]^n,\d[0,1]^n) \to (X,\ast)$} 
we can in principle use 
other
homotopy functors
analogously in place of~$\pi_1$,
getting representations of mapping class groups this
way.
In this paper, we will use the homotopy functor which sends
a  pointed space $(Y,\ast)$ to the triple $\pi_{(1,2)}(Y,\ast)$ defined as
$(\pi_1(Y,\ast),\pi_2(Y,\ast),\triangleright_{\pi_1})$, where
$\triangleright_{\pi_1}$ is the usual action of $\pi_1$ on~$\pi_2$.
The underlying algebraic notion is that of a
\emph{\agg}, see for example~\cite[Chapter 1, \S 1]{Baues4D}.
\Anagg\ is a triple~$\G=(G,A,\trr)$, where $G$ is a group, $A$ is
an abelian group, and $\trr$ is a left-action by automorphisms of~$G$ 
on~$A$;
see Subsection~\ref{SS:aaagg}.
A morphism of  \agg s is
a pair of group morphisms that respect the actions.

Artin's representation of the braid group~\cite{Artin},
and Dahm homomorphism for the extended loop braid group~\cite{Dahm, Goldsmith} can be seen as a 
special case of this construction. Let us recap here these two maps.

\subsection{Artin representation and Dahm homomorphism}
\label{SS:Artin_and_Dahm}

For~$\nnn \geq 1$, 
let $D^\nnn$ be the $\nnn$-disc~$[0,1]^\nnn \subset \R^\nnn$. 
Fix  
$d_\nn =\{ p_1,...,p_\nn \} \subset D^2 \setminus \partial D^2$
a set of  $\nn$ points in the interior of~$D^2$, 
and  
 {$\C = C_1\cup... \cup C_\nn  \subset D^3 \setminus \partial D^3$}
a set {made out of the union} of $\nn$ disjoint, unknotted, oriented circles, 
that form a trivial link with $\nn$ components 
in the interior of~$D^3$.

The \emph{braid group} $\BB\nn$ 
can be defined as the mapping class group of the pair $(D^2, d_\nn)$~(for a 
survey, see for instance~\cite{Birman_Brendle}).
Analogously the \emph{extended loop braid group}~$\LBE\nn$ 
is defined  as the mapping class group~$\MCG(D^3,\C)$, 
meaning the group of self-homeomorphism of the pair~$(D^3,\C)$, 
relative to the boundary of~$D^3$,
up to~$(D^3,\C)$-isotopy~\cite{Damiani}. 
Note
that this definition appears also in~\cite{Goldsmith}, 
in terms of \emph{motion groups}. Homeomorphisms $(D^3,\C)\to (D^3,\C)$  do not 
necessarily preserve the  orientation of~$\C$. If we consider the group 
$\MCG(D^3,\widehat{\C})$ of isotopy classes of those  homeomorphisms 
$(D^3,\C)\to (D^3,\C)$ that preserve also the orientation on~$\C$,  the group 
obtained is the \emph{loop braid group},
denoted by~$\LB\nn$.
The nomenclature ``loop braid groups'' is due to Lin~\cite{Lin:2008}.

Let $F_\nn$ be the free group 
of rank~$\nn$
generated by~$\{x_1, \ldots, x_\nn\}$. Then we have that $\pi_1(D^2 \setminus 
d_\nn, \ast)$ is isomorphic to~$F_\nn$. Thus, in the case of the 
pair~$(D^2,d_n)$,
Equation~\eqref{eq:MtoAut} becomes:
\[
\pn \colon B_n \to \Aut(F_n) .
\]
Artin shows~\cite{Artin} that this is an injection.

We also have that~$\pi_1(D^3 \setminus \C,\ast)
\cong F_n$, so
the map~$\pn \colon \MCG(D^3,\C)\to \Aut(\pi_1(D^3 \setminus \C,*))$
becomes:
\[
\pn\colon \LBE\nn \to \Aut(F_n).
\]
This homomorphism  is proven injective in~\cite{Dahm} and published 
in~\cite[Theorem~5.2]{Goldsmith}.

\subsection{Our main result.}
The space $D^2\setminus d_n$ is aspherical, and therefore its homotopy type can 
be recovered from the fundamental group. However $D^3 \setminus \C$ is not 
aspherical, hence
$\pi_1$ ``forgets'' more about $D^3  \setminus \C$ than it does
about~$D^2 \setminus d_n$.  In this paper, we work with~$\pi_{(1,2)}(D^3 
\setminus \C,*)$, in the intent of retrieving some of
the lost homotopical information when passing from $D^3  \setminus \C$
to ~$\pi_{1}(D^3 \setminus \C,*)$. How much  homotopical information is retained 
is discussed in Subsection~\ref{remarks}.

In particular, we will prove here that
$\pi_{(1,2)}(D^3 \setminus \C ,\ast)$ is {a \aaagg} of
\emph{free type}\footnote{The nomenclature ``of free type'' is borrowed 
from~\cite[Definition 7.3.13]{Brown_et_al}.}.
In practice, with this we mean 
that~$\pi_{(1,2)}(D^3 \setminus \C ,\ast)= \big(\pi_1(D^3 \setminus \C ,\ast), 
\pi_2(D^3 \setminus \C ,\ast), \triangleright_{\pi_1}\big)$,
is isomorphic to~$\M_n=(F_n,M_n,\trr)$,
where $M_n$ is the free $\Z[F_n]$-module generated by~$\{K_1, \ldots, K_\nn\}$.
Given a \agg~$\G=(G,A,\trr)$, a morphism $f=(f^1,f^2)\colon \M_n \to \G$ is 
therefore determined by the images $f^1(x_i)$ and $f^2(K_i)$; see 
Subsection~\ref{SS:aaagg}.

In this paper we construct an inclusion~$\Pn\colon \LBE\nn \to \Aut(\M_n)$.

\subsection{Structure of the paper}
In Section~\ref{S:Automorph} we recall some notions that will be used throughout 
the paper. We also give a topological realisation of Artin's representation for 
the braid group, and a hands-on flavour of the lift we construct for the 
extended loop braid group. 
In Section~\ref{S:Abelian_gr_groups} we present Dahm and Goldsmith's
lift of Artin's representation for the extended loop braid
group~(Theorem~\ref{gold-theorem}). Then we introduce \agg s
(Definition~\ref{D:agg}) and describe a lift of Artin's representation
for extended loop braid groups in the \agg\ 
$\pi_{(1,2)}(D^3\setminus L_n , \ast)$ of the $3$-ball with a set of $n$ 
(unlinked and unknotted) circles 
excised from its interior (Lemma~\ref{L:Two_circles} and Theorem~\ref{main1}). 
This is our first main result. In Section~4 we formalise the topological 
construction of the considered representation, in the second main result of this 
paper (Theorem~\ref{main2}).

%
%


%

\section{On automorphisms of free groups and beyond} 
\label{S:Automorph}
Here we first recall some constructions that are standard, but which will
have useful lifts to higher dimensions later. 

If ${\mathcal C}$ is a concrete category, then
$F\colon {\mathcal C} \rightarrow \SET$
is the forgetful functor. 
Let $\Grp$ be the category of groups,
and define $\Aut(G) \subset \Grp(G,G)$
as the subset of invertible homomorphisms.
The left adjoint
$F^a \colon \SET \rightarrow \Grp$
takes a set to the free group on that set.
We consider sets of form~$x_{\underline{\nn}}$ defined to be~$\{ x_1 , x_2 , ..., x_\nn \}$,
$\nn \in \N_0$
as a skeleton in the full subcategory $\FINSET$ of finite sets,
and define
\begin{equation}
\label{eq:deFn}
F_\nn = F^a(  \{ x_1, ..., x_\nn \}).
\end{equation}

By the adjoint functor property (``freeness'') elements of
$\Grp(F_n, G )$
are  
uniquely specified
by giving the image of each~$x_i$.
For example, for $\sigma\in \Sym(\nn)$, the symmetric group on~$\{1,\dots,n\}$,
define $t_\sigma \in \Aut(F_\nn)$ by:
\begin{equation}
\label{eq:sigma}
t_\sigma   \colon   x_i \mapsto x_{\sigma(i)}.
\end{equation} 
A \emph{non-automorphism} example in $\Grp(F_\nn,F_\nn)$ is given by~$x_i
\mapsto x_1$.
We use cycle notation for elements of~$\Sym(\nn)$,
thus we have $t_{(12)}$ and so on; 
and define
the \emph{translation automorphism} by~$t = t_{(1,2,...,\nn)}$. 
We say an automorphism is \emph{ \Cg\ } (CG)
if it acts trivially on all except at
most two adjacent~$x_i, x_{i+1}$, and on these it acts to produce
elements of~$F^a(x_i , x_{i+1})$.
We say this action is local at~$i,i+1$.

For example, the subset of automorphisms of form~\eqref{eq:sigma} are
not CG in general.
In particular
the translation automorphism $t$
is not CG.
However the subset forms a subgroup and the subgroup is
generated by 
the CG automorphism $t_{(12)}$ local at~$1$,~$2$
and $\nn-2$ translates thereof.

\subsection{On automorphisms of free groups realised topologically}\label{2dcase}
A topological realisation of $F_\nn$ is given by~$\pi_1(D^2\setminus 
d_\nn,\ast)$.
In order to explicitly write down
an Artin representation~$\pn \colon B_n \to \Aut(F_n)$, 
we specify $d_n$ and a free basis for~$\pi_1(D^2\setminus d_\nn, \ast)$. 
We take the points
$d_\nn = \{ p_1,\dots, p_n \}$
to be along a horizontal line. 
We consider paths~$\widehat{x_i}$, $i=1,\dots,n$  passing clockwise 
around each~$p_i$, as in Figure~\ref{Xms}, 
such that their images intersect only at~$\ast \in \d D^2$. 
Let $x_i$ be the homotopy class of~$\widehat{x_i}$. 

\begin{figure}[thb]
\centering
\includegraphics[scale=0.6]{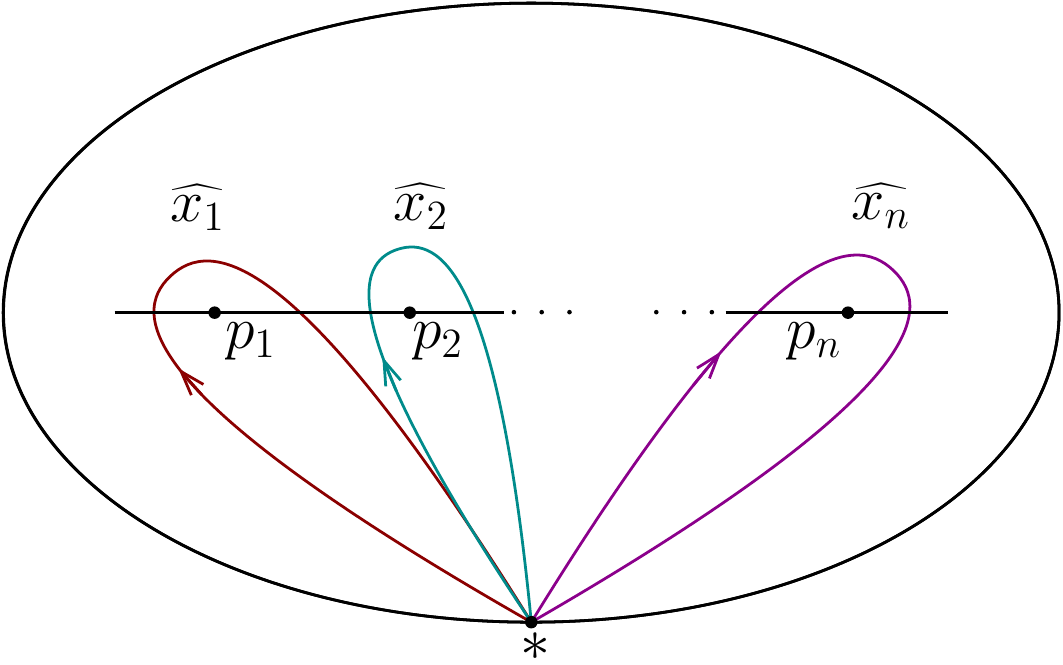}
\caption{A free basis for~$\pi_1(D^2\setminus d_\nn, \ast)$.} 
\label{Xms} 
\end{figure}

Note that 
$\pi_1(D^2\setminus d_\nn, \ast)$ is free on~$\{x_1,\dots, x_n\}$. 
Also, 
$x_1x_2\dots x_n$ is 
the homotopy class of  a path that traces the boundary of~$D^2$, 
clockwise, starting and finishing at~$\ast$.

Note also that $(D^2\setminus d_\nn,\ast)$ strongly deformation retracts 
into the pointed subspace made of the images of the paths, 
which is homeomorphic to~$\bigvee_{i=1}^n (S^1, \star)$.
Combining with Seifert -- van Kampen theorem, this implies that  $\pi_1(D^2\setminus d_\nn,\ast)$ is freely
generated by~$\{x_1,\dots,x_n\}$. 
The same type of argument
will be used when we address the higher case in 
Section~\ref{S:topological_interpretation}, and prove that~$\pi_{(1,2)}(D^3\setminus \C,\ast)$ is of free type.

Using Equation~\eqref{Int_1_1}, elements of $\MCG(D^2,d_\nn)$ induce elements of~$\Aut(F_\nn)$.
Consider for example the mapping class $\Sigma_1$ indicated by \textbf{a) - c)} of Figure~\ref{fig:taut1}.
Note that
this corresponds to the $1$, $2$ local automorphism
$\mathcal{S}_1^1 {\colon F_n \to F_n}$
given by 
\begin{equation}
\label{eq:xbd}
x_1 \mapsto x_2, \hspace{1in}
x_2 \mapsto x_2^{-1} x_1 x_2.
\end{equation}

\begin{figure}[b]
\centering
\includegraphics[scale=0.6]{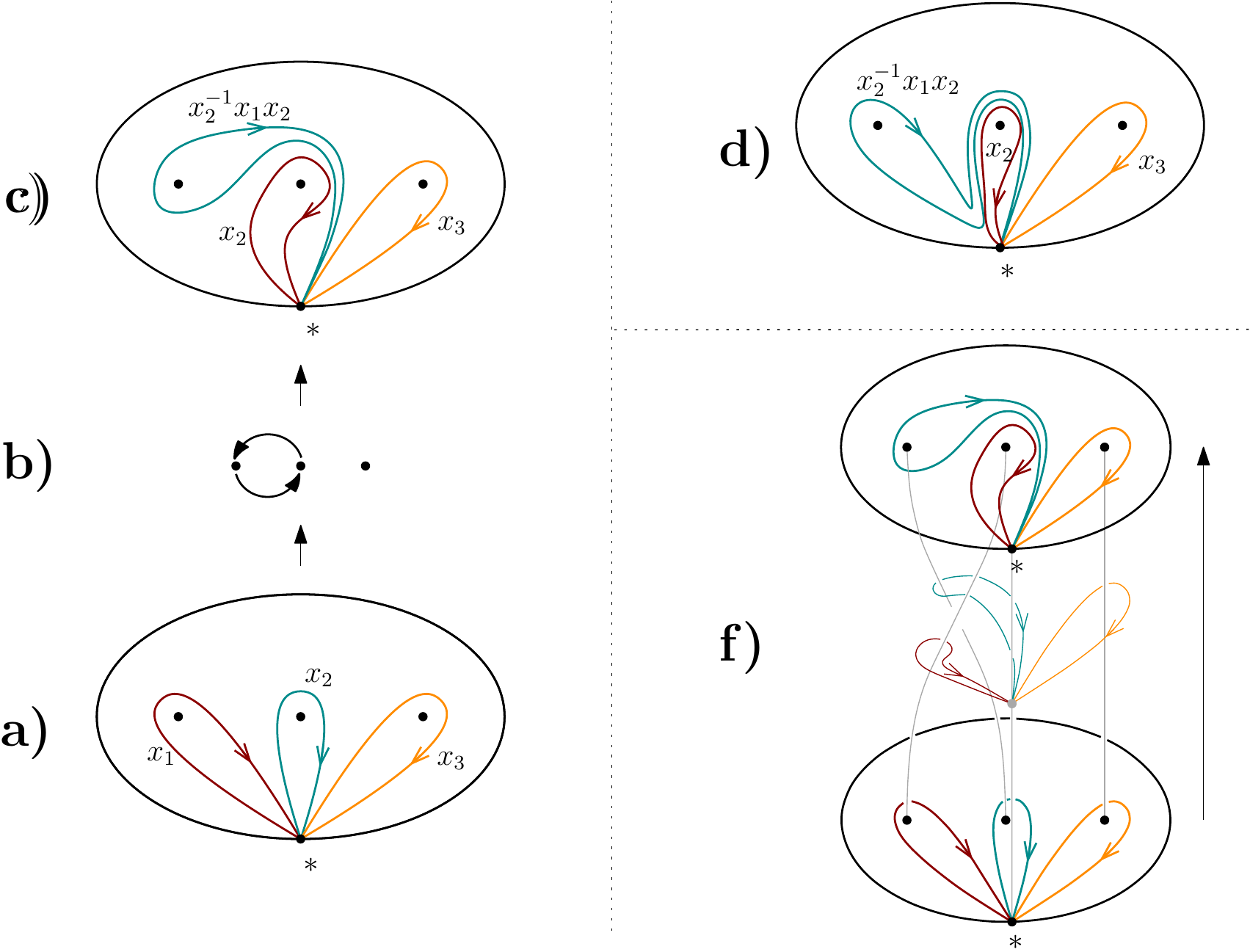}
\caption{\textbf{a) - c)}~An exchange of the marked points inducing a homeomorphism of the disk, and the corresponding automorphism 
of $\pi_1 (D^2\setminus d_3 , \ast)$. \textbf{d)}~The new paths are expressed in the original basis: thus $x_1$ has morphed to $x_2$ and $x_2$ has morphed to $x_2^{-1} x_1 x_2$. \textbf{f)}~The loops following the \emph{timelines} of the braiding. 
}
\label{fig:taut1}
\end{figure}

Define
$\mathcal{S}_i^1 {\colon F_n \to F_n}$
as the translate of~$\mathcal{S}_1^1$;
an explicit formula is in Equation~\eqref{defofS1A}.
From the topological realisation (or direct calculation) we have:
\begin{equation}
\label{eq:YBE}
\mathcal{S}^1_i \circ \mathcal{S}^1_{i+1}\circ \mathcal{S}^1_i = \mathcal{S}^1_{i+1}\circ \mathcal{S}^1_i\circ \mathcal{S}^1_{i+1}.
\end{equation}

\subsection{Towards automorphisms in higher dimension}
\label{SS:towards higher}

Our task here is to describe automorphisms of a suitable lift
of the free group to ``higher dimension''. Here  dimension refers
to the topological interpretation of $F_n$ as~$\pi_1(D^2 \setminus d_n,\ast)$,
whose building blocks are paths in the point-punctured disk~$D^2 \setminus d_n$.

The lift involves the triple~$\pi_{(1,2)}(D^3 \setminus \C ,\ast)= \big(\pi_1(D^3 \setminus \C
,\ast), \pi_2(D^3 \setminus \C ,\ast), \triangleright_{\pi_1}\big)$, where
$\triangleright_{\pi_1}$ denotes the usual action of $\pi_1$ on~$\pi_2$.

In the $F_n$ case, freeness means that
an automorphism 
is determined by the images of
a set of free group generators. 
In the lift that we consider, the corresponding structure $\pi_{(1,2)}(D^3\setminus \C)$ is not 
a free group, but a \agg\ of free type.
In particular, as we will prove in Subsection~\ref{SS:aaagg}, $\pi_2(D^3\setminus \C,\ast)$ is  a free $\Z$-module.
In order to understand the basis we make some preparations.

Given a pointed topological space~$(X,\ast)$, then the action of $\pi_1(X,\ast)$ on~$\pi_2(X , \ast)$, means that~$\pi_2(X , \ast)$,
can naturally be equipped with the 
structure of~${\Z[\pi_1(X ,\ast)]}$-module.
We will give a visual idea
of this structure for the case~$X= D^3\setminus \C$.

We note that the \textit{``balloons and hoops''} point of view we present in 
this subsection for the visual description of~$\pi_{(1,2}(D^3\setminus \C,*)$ is 
essentially as in~\cite{Ballons}, where the term is coined. This balloons and 
hoops approach for understanding~$\pi_{(1,2)}(D^3\setminus \C,*)$  was also used 
in~\cite[\S 4.5.1]{loopy}.

Note that
there is not a canonical choice of the elements considered to be
generating for example in~$\pi_1(D^3\setminus L_n,*)$.
Varying the precise choice of $L_n$ satisfying the
``unlinked circles'' characterisation
affects what might be considered as a natural choice,
and hence affects the construction
--- albeit only up to isomorphism.
Later it will be convenient to work with $L_n$ a row of circles
confined to a plane in~$D^3$.
But in this section we will instead use for $L_n$ a stack of circles
confined to a single axis of rotational symmetry. 

Note that by rotating $[0,1]^2 \times   \{0\}$ about
$\{ 0 \} \times [0,1]\times\{0\}$
we obtain a topological~$D^3$, and this rotation
causes 
the points 
$p_1,\dots,p_n$ to sweep out an~$\C=C_1\cup \dots \cup C_n$,
one where the circles are ``stacked'' coaxially:
see Figure~\ref{fig:rot13}, subfigures~\textbf{a)},~\textbf{b)}.
Note that we also assume that the circles $C_1,\dots,C_n$ are stacked on top of each others, in decreasing height.

\begin{figure}[htb]
\centering
\includegraphics[scale=0.6]{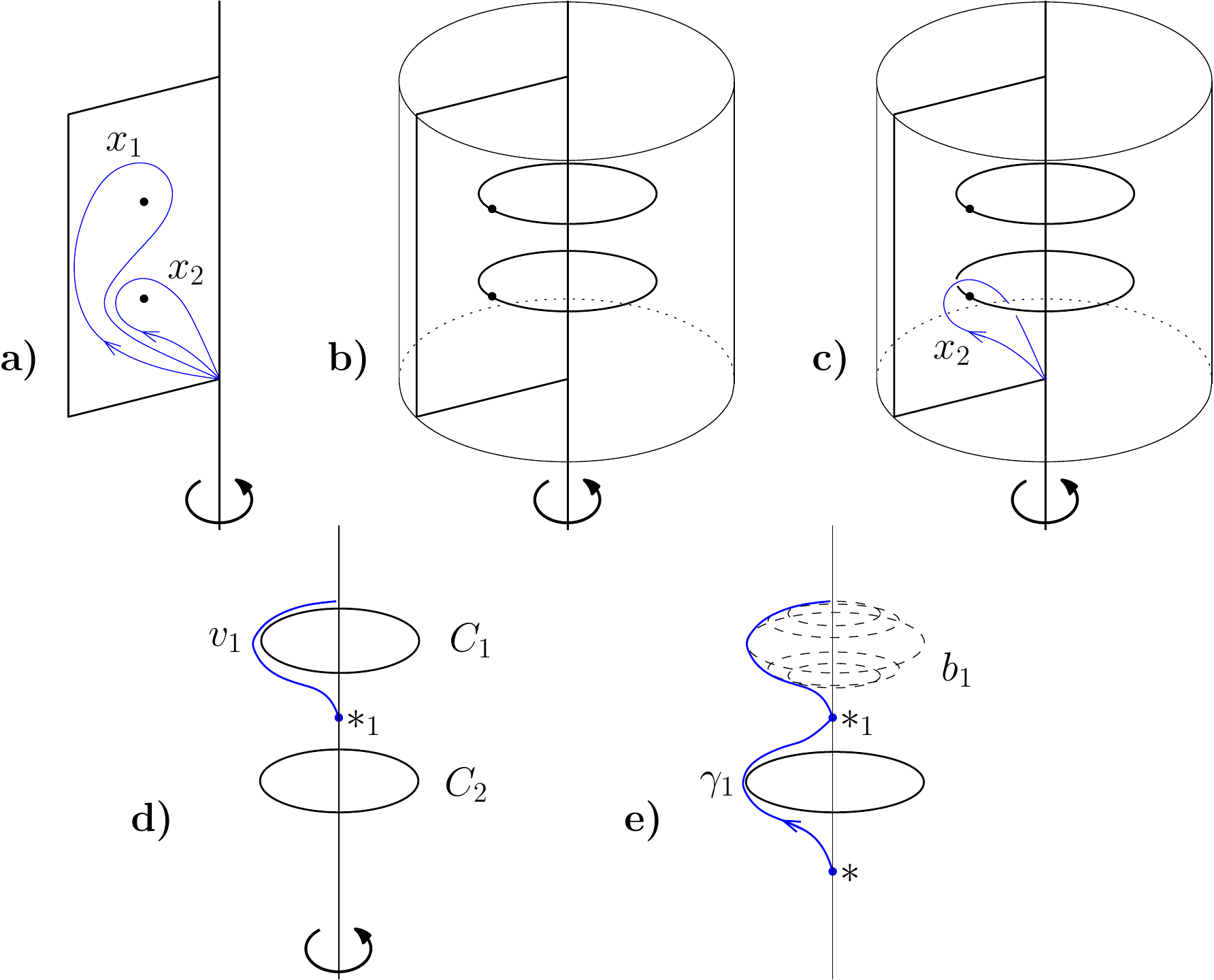}
\caption{\textbf{a), b), c)} Rotating the two-punctured disk into~$\R^3$. \textbf{d), e)}  
constructing a sphere (a ``balloon'') $b_1$, enclosing $C_1$, by rotating the path 
$v_1$. The element $K_1\in \pi_2(D^3 \setminus L_n,\ast )$ is obtained from the 
homotopy class $K'_1\in  \pi_2(D^3 \setminus L_n, \ast_1 )$ of a positively oriented 
parametrisation $(D^2,\d D^2) \to (b_1,\ast_1)$, of $b_1$,  made into an element $K_1\in 
\pi_2(D^3 \setminus L_n,\ast)$ by using the isomorphism $\pi_2(D^3 \setminus L_n,\ast_1 )\to \pi_2(D^3 \setminus L_n,\ast )$ derived from  the path $\gamma_1$.}
\label{fig:rot13} 
\end{figure}

The group $\pi_1(D^3 \setminus L_n,\ast)\cong F_n$ is freely generated by
(classes of) loops $x_i$
  around the circles as illustrated in Figure~\ref{fig:rot13}\textbf{(c)}.  These 
classes have representatives that lie in the original copy of~$D^2$.
The image of the element called $x_i$ in $\pi_1(D^2\setminus d_n,\ast)$, as in   Figure~\ref{fig:rot13}\textbf{(a)}, we
again call~$x_i$.

Now let $g$ be a homeomorphism {representing a certain class} in~$\MCG(D^2,d_\nn)$.
This {$g$} restricts to~$(0,1]\times [0,1] $:  
the omitted edge is the rotation axis, so 
rotating gives a self-homomorphism of~$(0,1]\times [0,1] \times S^1$.
Topologically, $(0,1]\times [0,1] \times S^1$ is
$D^3$ with the axis removed.
The {latter homeomorphism}  thus extends to a self-homomorphism $r(g)$ of $D^3$ by inserting
the constant function on the axis.
This construction lifts to a well-defined {group homomorphism}
\[
r \colon \MCG(D^2,d_n) \rightarrow \MCG(D^3,L_n).
\]

In particular, consider the 
mapping class $\Sigma_i$ in $\MCG(D^2,d_n)$
exchanging consecutive points $p_i$ and $p_{i+1}$, as in Figure~\ref{fig:taut1}, which considers the case $n=3$ and~$i=1$.
Its image under ~$r$, which we also denote~$\Sigma_i$,
corresponds to  an exchange of circles,
as illustrated in Figure~\ref{F:loopin0}, for $n=2$ and~$i=1$.
To the mapping class $\Sigma_i\in \MCG(D^3,\C)$ we will call \emph{elementary braid permutation}.

\begin{figure}
  \includegraphics[width=3.3cm]{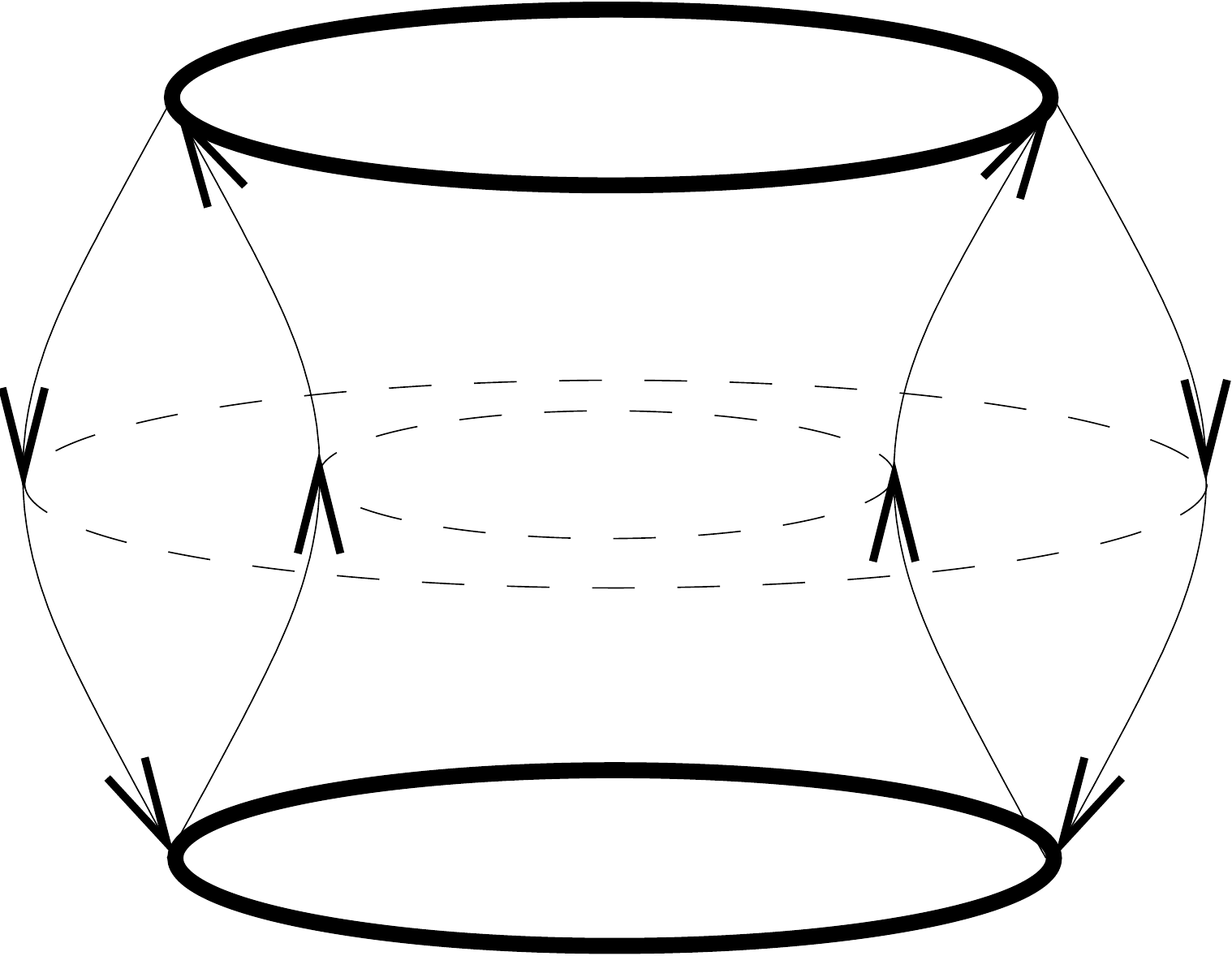}
\caption{An ``exchange'' of coaxial circles,
 where the circle $C_1$ is originally on top.
  This  is a smooth family of embeddings of  $S^1\sqcup S^1$ into~$D^3$.
  It induces a diffeotopy $t \in [0,1] \mapsto \phi_t \in \mathrm{Homeo}(D^3)$ of~$D^3$, relative to the boundary, by the isotopy extension
theorem (either before or after applying~$r$).
The end-value $\phi_1\in \mathrm{Homeo}(D^3)$ of the isotopy is a homeomorphism~$(D^3,L_2) \to (D^3,L_2)$.
The elementary braid permutation $\Sigma_1\in \MCG(D^3,L_2)$ is the mapping class of~$\phi_1$.
\label{F:loopin0}} 
\end{figure}

One  new aspect in the higher setting is that there are elements of
$\MCG(D^3,L_n)$ that do not have representatives with the rotational
symmetry, such as maps that exchange the circles by taking them
off the rotation axis.
Note that these induce elements of the automorphism group $\Aut( \pi_1(D^3 \setminus L_n) )$ of type~$t_{(i,i+1)}$.
Also breaking the axial symmetry in this sense, are motions that flip
a single circle {$C_i$} onto itself.
One can see that this induces an automorphism local at~$i$, given by
$T \colon x_i \mapsto x_i^{-1}$.

Another  new aspect in the higher setting is that
$\pi_2(D^3 \setminus L_n, \ast)$  is not
trivial. An example of a non-trivial element of
$\pi_2(D^3 \setminus L_n,\ast_i)$ 
is the homotopy class 
of a 
\emph{wrapping square} $K_i'\colon (D^2,\d D^2) \to (D^3 \setminus L_n,\ast_i)$,
based at $\ast_i$ (a point in the rotating axis) that wraps a single circle~$C_i$, including the disk of which ${C_i}$ is boundary,
exactly once,
see Subfigures~{\textbf{d)} and \textbf{e)}} in
Figure~\ref{fig:rot13} for the {$i,n=2$} case. Concretely $K_i'$
is a positively oriented parametrisation of a \emph{balloon} -- i.e. a $2$-sphere $b_i$ containing~$C_i$,
and no other circle~$C_j$, oriented by an exterior normal. If
we connect $\ast_i$ to $\ast$ by a path $\gamma_i$ that does not cross the
disks spanned by the circles in~$\C$, and consider the usual
isomorphism $\pi_2  (D^3 \setminus L_n,\ast_i) \to \pi_2  (D^3
\setminus L_n,\ast)$ derived from $\gamma_i$~{(see e.g.~\cite[Page 343]{Hatcher})},
this gives a non-trivial element $K_i =[\gamma \triangleright K_i'] \in \pi_2(D^3 \setminus L_n,\ast)$.

We argue in Section \ref{S:topological_interpretation},
  that $\pi_2(D^3 \setminus \C,\ast)$ is freely generated, as an abelian group, by the 
elements $p \trr K_i$, where $i \in \{1, \dots,n\}$ and $p \in
\pi_1(D^3 \setminus \C,\ast)$. Each $p \trr K_i$ can be visualised as a
hoop~$p$, connecting $\ast$ to~$\ast$, which is then attached to~$K_i$.
In Figure~\ref{fig:OO1x}
we show some examples of elements of~$\pi_2(D^3\setminus L_2, \ast)$.

\begin{figure}[htb]
\centering
  \includegraphics[scale=0.6]{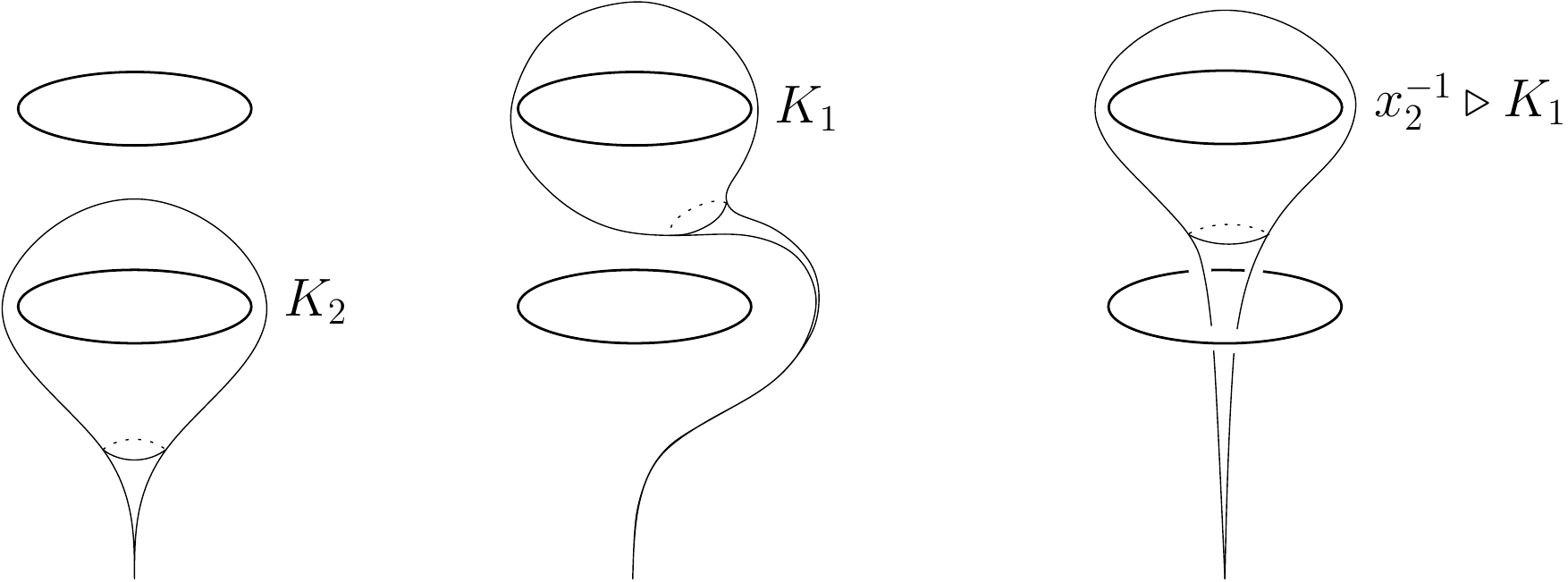}
\caption{Elements in $ {\pi_2(D^3 \setminus L_2,\ast)}$.
  We number the top circle {1}.
Note that the difference between the second and third figures is
given by the path $x_2^{-1}$ around circle~${C_2}.$}
\label{fig:OO1x} 
\end{figure}

A crucial property of the classes $K_1,\dots,K_n \in \pi_2(D^3\setminus \C,\ast)$ is that  $K_1+K_2+\dots+K_n$ can be represented by a parametrisation~$(D^2,\d D^2) \to (\d D^3,\ast)$. 
With these points  in mind, we lift consideration of the mapping class group
action on
the fundamental group to an
action of the mapping class group $\MCG(D^3,\C)$
 on the 
triple:
\[
  {\pi_{(1,2)}(D^3\setminus L_n , \ast)= (\pi_1(D^3\setminus L_n, \ast), \pi_2(D^3\setminus L_n, \ast), \triangleright_\pi)},\]
where $\triangleright_\pi$ denotes the usual action of  $\pi_1(D^3\setminus L_n, \ast)$ on~$\pi_2(D^3\setminus L_n, \ast)$.
Since {$\pi_2(D^3\setminus L_n, \ast)$} is abelian,
$\pi_{(1,2)}(D^3\setminus L_n, \ast)$
is a  \aaagg, a structure we will define in Section~\ref{SS:aaagg} .

\section{The \TAR}
\label{S:Abelian_gr_groups}

We discussed in Subsection~\ref{SS:Artin_and_Dahm} that
the homotopical information coming from
$\pi_2(D^3\setminus L_n, \ast)$  is not taken 
into account by Dahm's lift of Artin's representation.  With this in 
mind, we proceed to define a
\emph{\TAR}
for extended loop braid groups~$\LBE\nn$.
 The codomain for our representation stores two levels of information,  coming from the first and  second homotopy groups of~$D^3
 \setminus \C$, as well as the action of the first on the former.
It will thus be a \emph{\agg}. 
In this Section, we recap some results about
extended loop braid groups and \agg s, which will subsequently be used to define the~\TAR.

\subsection{Loop braid groups and extended loop braid groups}
\label{SS:review}

We recall~\cite{Artin} that the braid group $\BB\nn = \MCG(D^2,d_n)$ is isomorphic to the group defined by generators
$\{\sigma_1,\dots, 
\sigma_{n-1}\}$, subject to relations:
\begin{subequations}
\begin{align}
\sigma_i\sigma_j & =\sigma_j\sigma_i, &\textrm{ for } |i - j| > 1; \label{bg1}\\
\sigma_i\sigma_{i+1}\sigma_i &=\sigma_{i+1}\sigma_{i}\sigma_{i+1}, &\textrm{ for 
} i = 1, \dots , n - 2. \label{bg2}
\end{align}
\end{subequations}

The generator $\sigma_i$ is the image of the mapping class $\Sigma_i$ of
$\MCG(D^2,d_n)$ exchanging the $i$th and $i+1$th point. 

Passing to the automorphisms of the free group $F_\nn$ of rank~$\nn$, let $\mathcal{S}_i^1\colon F_n \to F_n$ be
defined on the generators $\{x_1,\dots,x_n\}$ as:
\begin{equation}
\label{defofS1A}
\mathcal{S}_i^1 \colon
\begin{cases}
	x_i \mapsto  x_{i+1}, \\ 
	x_{i+1} \mapsto  x_{i+1}^{-1} \, x_i \,x_{i+1}, \\ 
	x_j \mapsto  x_j, \textrm{ if } j<i \textrm{ or } j>i+1\\
\end{cases}.
\end{equation}

Note from Figure~\ref{fig:taut1} that 
Artin's representation
$\pn\colon \BB\nn \to \Aut(F_n)$ is such that, for 
$i=1,\dots,n-1$,~$\Sigma_i \mapsto \mathcal{S}_i^1$, and this data determines~$\pn$.

\begin{Theorem}[{\cite{Artin}, for a more recent proof see for 
instance~\cite[Theorem 5.1]{Hansen}}]
\label{artin_theorem}
The map
$\pn\colon \BB\nn \to \Aut(F_n)$ is injective, and an automorphism 
$\phi$ of~$F_\nn$ is in $\pn(\BB\nn)$ if and only if the following two 
conditions are satisfied:
\begin{enumerate}
\item \label{conjugating} There exist $a_1,\dots, a_n \in F_n$, and a 
permutation $\alpha$ of $\{1,\dots, n\}$, such that  $\phi(x_i)=a_i \, 
x_{\alpha(i)}\, a_i^{-1}$. 
\item \label{prod-of-gens} $\phi(x_1 \, x_2 \dots x_n)=x_1 \, x_2 \dots x_n$. 
\end{enumerate}
\end{Theorem}
Note that \textit{\eqref{prod-of-gens}} holds since $x_1x_2\dots x_n$ is  
the homotopy class of  a path that traces the boundary of $D^2$, 
clockwise, starting and finishing at $*$. Hence $x_1x_2\dots x_n$ is  left untouched  by all elements of $\MCG(D^2,d_\nn)$, since they are required to be the identity on~$\d D^2$.
\medskip

Let us move one dimension up, in the realm of extended loop braid
groups~$\LBE\nn$, defined as the mapping class groups~$\MCG(D^3,\C)$,
 and loop braid groups~$\LB\nn$, defined as
$\MCG(D^3,\widehat{\C})$; see Section~\ref{SS:Artin_and_Dahm}).
In this case we have the following result \cite{Brendle_Hatcher, Baez_et_al, Damiani}:
\begin{Theorem}\label{gensandrelsforLBG} 
{The  group~$\LB\nn$ is isomorphic to  the abstract presented group defined by generators
$\{\sigma_i , \rho_i \mid  i = 1, \dots , n - 1\} $
subject to relations:}
\begin{subequations}
\begin{align}
\sigma_i\sigma_j & =\sigma_j\sigma_i, &\textrm{ for } |i - j| > 1 
\label{LBG1};\\
\sigma_i\sigma_{i+1}\sigma_i &=\sigma_{i+1}\sigma_{i}\sigma_{i+1}, &\textrm{ for 
} i = 1, \dots , n - 2;\\
\rho_i
\rho_j & = \rho_j \rho_i, &\textrm {for } |i - j| > 1;\\
\rho_i \rho_{i+1}\rho_i &= \rho_{i+1}\rho_i \rho_{i+1}, &\textrm{ for } i = 1, 
\dots , n - 2;\\
\rho_i^2 &= \id, &\textrm{ for } i = 1, \dots , n - 1;\\
\rho_i\sigma_j &=\sigma_j \rho_i , &\textrm{ for } |i - j| > 1;\\
\rho_{i+1}\rho_{i}\sigma_{i+1} &=\sigma_i \rho_{i+1}\rho_i &\textrm{ for } i = 
1, \dots , n - 2;\\
\sigma_{i+1}\sigma_i \rho_{i+1} &= \rho_i\sigma_{i+1}\sigma_i,& \textrm{ for }  
i = 1, \dots , n - 2.\label{LBG8}
\end{align}
\end{subequations}
Moreover the  group $\LBE\nn$ is isomorphic to the abstract presented group defined by generators
$\{\sigma_i , \rho_i
\mid i = 1, \dots , n - 1\} \cup \{\tau_j \, |\, j=1, \dots, n\}$
subject to relations \eqref{LBG1} to \eqref{LBG8} above, together with:
\begin{subequations}
\begin{align}
\tau_i \tau_j &= \tau_j \tau_i, &\textrm{ for } i\neq j;\label{eLBG1}\\
\tau_i^2&=\id, &\textrm{ for } i=1,\dots, n;\\
\sigma_i \tau_j &= \tau_j\sigma_i, &\textrm{ for } |i - j| > 1;\\
\rho_i \tau_j &= \tau_j\rho_i, &\textrm{ for } |i - j| > 1;\\
\tau_i\rho_i &= \rho_i \tau_{i+1}, & \textrm{ for } i = 1, \dots , n - 1;\\
\tau_i\sigma_i &= \sigma_i \tau_{i+1}, &\textrm{ for } i = 1, \dots , n - 1;\\
\tau_{i+1}\sigma_i &= \rho_i\sigma_i^{-1} \rho_{i}\tau_{i},& \textrm{ for } i  = 
1, \dots , n - 1.  \label{eLBG7}
\end{align}
\end{subequations}
\end{Theorem}

Let $\Sigma_i$ denote the mapping class corresponding to the {elementary braiding permutation} in~$\MCG(D^3,L_n)$.
In the coaxial configuration of~$L_n$,~$n=2$, we may represent this as in
Figure~\ref{F:loopin0}.
Similarly, let $\Rho_i$ denote the mapping class of the \textit{non-braiding permutation}; and $\Tau_i$ be the mapping class corresponding to a 180 degrees flip of $C_i$ with respect to the vertical axis, see Figure~\ref{F:flips}.
The isomorphism between $\LBE\nn$ and the abstract presented group
sends~$\Sigma_i \mapsto \sigma_i$,~$\Rho_i \mapsto \rho_i$, and~$\Tau_i \mapsto \tau_i$.

\begin{figure}[hbt]
\centering
\includegraphics[scale=0.6]{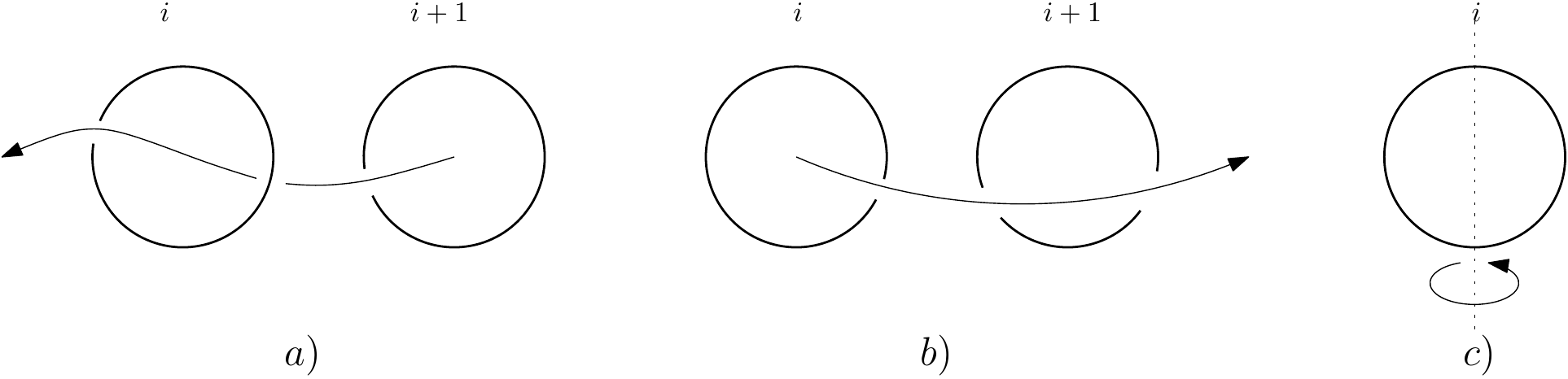}
\caption{$a)$ Pictorial idea of the mapping class~$\Sigma_\ii$. $b)$ The mapping class~$\Rho_\ii$.  $c)$ The mapping class~$\Tau_\ii$.}
\label{F:flips}
\end{figure}

Dahm~\cite{Dahm} generalised Artin's representation to general manifolds with a  compact submanifold in its interior. Goldsmith published his result for the case 
of the pair~$(\R^3, \C)$, which gives the same group as the pair~$(D^3, \C)$. 
Let us recall it: 

\begin{Theorem}[{\cite[Theorem 5.3]{Goldsmith}}]
\label{gold-theorem} 
For~$\nn \geq 1$, the map
\[
\pn \colon  \LBE\nn \longrightarrow \Aut(F_n).
\]
is injective.
Its image is the subgroup of $\Aut(F_\nn)$ consisting of automorphisms of the 
form $x_\ii \mapsto a_\ii\inv x^{\pm 1}_{\alpha({\ii})} a_\ii$ where $\alpha$ 
is a permutation of $\{1,\dots,\nn\}$ and~$a_\ii\in F_n$.
Moreover, this subgroup of $\Aut(F_\nn)$ is generated by the automorphisms
$\{\mathcal{S}_\ii^1 \mid i=1, \dots, \nn-1\}$, $\{\mathcal{R}_\ii^1 \mid i=1, \dots, \nn-1\}$ and 
$\{\mathcal{T}_\ii^1 \mid i=1, \dots, \nn\}$,  where  $\mathcal{S}_i^1$ is as in Equation~\eqref{defofS1A}, 
and: 
\begin{align}
\label{E:R}
\mathcal{R}_\ii^1 &: \begin{cases}
            x_\ii \mapsto x_{\ii+1}; & \\
            x_{\ii+1} \mapsto x_\ii; &\\
            x_\jj \mapsto x_\jj, \ &\text{for} \ \jj \neq \ii, \ii+1.
		\end{cases} \\
\label{E:T}
\mathcal{T}_\ii^1 &: \begin{cases}
            x_\ii \mapsto x\inv_\ii; &\\
            x_\jj \mapsto x_\jj, \ &\text{for} \ \jj \neq \ii.
        \end{cases}
\end{align}
\end{Theorem}

This case
of the map $\pn$ is known in the literature as \emph{Dahm's homomorphism}.
The 
generators of $\LBE\nn$ of type~$\sigma_\ii$, $\rho_\ii$ and~$\tau_\ii$ are 
respectively sent by Dahm's homomorphisms to automorphisms~$\mathcal{S}_\ii^1$,
$\mathcal{R}_\ii^1$ and~$\mathcal{T}_\ii^1$.

\begin{Remark}  
\label{R:consequences_on_loop_braids}
A consequence of Theorem~\ref{gold-theorem} is that 
$\LB\nn$ injects into~$\Aut(F_\nn)$ and its image is isomorphic to the group of  
automorphisms of the form~$ x_\ii \mapsto a_\ii\inv x_{\alpha({\ii})}  a_\ii$,
where $\alpha$ is a permutation and~$a_\ii\in F_n$.
This 
group is called \emph{the group of {basis} conjugating automorphisms} of~$\Aut(F_\nn)$, and is 
generated by the automorphisms $\{\mathcal{S}_\ii^1 \mid i=1, \dots \nn-1\}$ and~$\{\mathcal{R}_\ii^1 \mid i=1, \dots \nn-1\}$.
\end{Remark}

From now on we will  focus on the extended loop braid group, keeping in mind 
that through Remark~\ref{R:consequences_on_loop_braids} consequences of this 
work can be drawn on loop braid groups. 

Dahm's homomorphism can be constructed with the same classical construction of 
Artin's representation in terms of mapping classes. 
Briefly, choose a base-point $\ast$ in the boundary of~$D^3$. Isotopies of maps 
$f\colon(D^3,\C) \to (D^3,\C)$ do not move the base-point. If $[f] \in \LBE\nn$  
then $[f]$ yields a pointed-homotopy class
$\Theta([f]) \colon (D^3 \setminus 
\C,\ast)\to (D^3 \setminus \C,\ast)$. Let 
\begin{equation}
\label{deftheta'}
\theta'([f])\colon \pi_1(D^3 \setminus \C,\ast)   
\longrightarrow \pi_1(D^3 \setminus \C,\ast)      
\end{equation}
be the induced map on fundamental groups.          
Algebraic topological considerations give that~$\pn([f])$ coincides with~$\theta'([f])$.

\subsection{The category of \texorpdfstring{\agg s}{pi-modules}}
\label{SS:aaagg}

\begin{Definition}
\label{D:agg}
 \Anagg\
  is given by a triple~$\G=(G,A,\trr)$, where $G$ is a group, $A$ is an 
abelian group, and $\trr$ is a left-action by automorphisms of $G$
on~$A$. Given 
two \aaagg s  $\G=(G,A,\trr)$ and~$\G'=(G',A',\trr')$, a morphism 
$f=(f^1,f^2)\colon \G \to \G'$ is a pair
$(f^1\colon G \to G', f^2\colon A \to 
A')$
of homomorphisms that preserve actions: for each $g \in G$ and~$a \in A$, 
we have~$f^2(g \trr a)=f^1(g) \trr'  f^2(a)$.
\end{Definition}

Given two composable \aaagg\ morphisms,
$f=(f^1,f^2)$ and~$h = (h^1,h^2)$,
then our convention will be that~$(f^1,f^2) (h^1,h^2)=(f^1 
\circ h^1,f^2 \circ h^2)$.
It is easy to see that the set of \aaagg s and their morphisms form a
category.
We write $\aagg$ for the category of \aaagg s, since \agg s are abelian group--group pairs.
The category of \agg s was also used in \cite[Chapter I, \S 1 (1.7)]{Baues4D}. Confer also for example~\cite[\S8]{Auslander74},
and~\cite{Sieradski1993}
where they are termed ``second homotopy modules''. 

Consider the forgetful functor
${\UU}\colon\aagg \rightarrow {\SET\times \SET}$
that takes
$\G = (G,A,\triangleright)$ to the pair~$(G,A)$, made from the underlying sets of the groups $G$ and~$A$.
Recall $F_n = F^a(\{1,2,...,n\})$ from~\eqref{eq:deFn}:
the adjoint
$\FFa$
takes $(\{1,2,...,n\},\{1,2,...,n\})$ to  \anagg\ that we construct next.

\begin{Lemma}
A morphism $f=(f^1,f^2)\colon \G \to \G'$ of \aaagg s is invertible if and only if
both $f^1$ and $f^2$ are invertible homomorphisms.
\end{Lemma}
\proof
It follows from the definition of a morphism of \aaagg s.
\endproof

We denote the group of invertible morphisms {$\G \to \G$} by~$\Aut(\G)$.

\begin{Definition} 
\label{D:M_n}
Let $\Z [G]$ denote the group ring of group~$G$.
Let $M_n^G$ be the free 
$\Z[G]$-module on the symbols~$\{K_1,\dots,K_n\}$, hence
\[
M_n^G  =\Z[G]\{K_1,\dots,K_n\}\cong\bigoplus_{i=1}^n \Z[G]{K_i}\cong 
\bigoplus_{g \in G}\bigoplus_{i=1}^n \Z(g,K_i)
\]
equipped with the diagonal action of~$\Z[G]$.
We define $\M_n^G$ to be the \agg~$(G,M_n^G,\trr)$.
The action of $G$ on $M_n^G$ is induced by the $\Z[G]$-module structure.
Define
\[
M_n = M_n^{F_n}, \hspace{1cm} 
\M_n^m = \M_n^{F_m} \hspace{1cm} \mbox{and} \hspace{1cm}
\M_n = \M_n^n . 
\]
\end{Definition}

The following proposition describes how morphisms $\M_n \to \G$ can be uniquely 
specified by their value on the generators $x_i$ and~$K_i$, for~$i=1, \dots, \nn$. This result is also used, implicitly, in~\cite[\S 4.5.1]{loopy}.

\newcommand{\btrr}{\blacktriangleright} 

\begin{Proposition}
\label{extfromgens}
Let $\G=(G,A,\btrr)$ be a \aaagg. There is a canonical one-to-one correspondence 
between morphisms  $f=(f^1,f^2) \colon \M_n^m \to \G$ and pairs of tuples 
$(g_1,\dots,g_m) \in G^m$ and~$(a_1,\dots,a_n) \in A^n$. The correspondence is 
such that:
$g_i=f^1(x_i)$ and~$a_i=f^2(K_i)$. 
\end{Proposition}

\begin{proof}
Let us consider a pair of tuples $(g_1,\dots,g_m) \in G^m$ and~$(a_1,\dots,a_n) \in A^n$. 
We associate to this pair a morphism $f=(f^1,f^2)$ as follows. 
We take $f^1$ to be the unique homomorphism $F_m \to G$ such that~$f^1(x_i)=g_i$.
As an abelian group, $M_n^m$ is free on the set of pairs $(g,K_i)$, where $g \in 
F_m$ and $i=1, \dots, n$. We define $f^2\colon M_n^m \to A$ to be the unique 
homomorphism such that~$f^2(g,K_i)= f^1(g) \btrr a_i$.
Compatibility with group 
actions is preserved by construction.
\end{proof}

\begin{Remark}
  Because of Proposition~\ref{extfromgens}, we say that
$\M_n^m$, and in particular~$\M_n$,  is 
\anagg\ of 
\emph{free type}.
Indeed consider 
the forgetful functor 
$\UU\colon \aagg \rightarrow {\SET\times \SET}$ that takes
$\G = (G,A,\triangleright)$ to~$(G,A)$.
This functor has a left adjoint~$\FFa$, where given a pair $(X,Y)$ of sets,~$\FFa(X,Y)=(F(X),\Z[F(X)]Y,\trr)$, where
$F(X)$ is the free group on~$X$, and $\Z[F(X)]Y$ is the free $\Z[F(X)]$-module on~$Y$, with the obvious action of~$F(X)$. Hence:
\[
\M_n^m \cong \FFa \redx{\big(\{1,\dots,m \},\{1,\dots,n\}\big)} .
\]
We say ``of free type'' instead of ``free'',
as $\M_n$ does not arise from a 
left adjoint of a forgetful functor from
$\aagg$  
to~$\SET$.
\end{Remark}

We observed in Section~\ref{SS:towards higher}
that $\pi_1(D^2\setminus d_n , \ast)$ is a free group~$F^a( \{ x_1,...,x_n \})$.
As we will establish in Lemma~\ref{ultimate_freeness}~\textit{et seq.},
we have~$\pi_{{(1,2)}}(D^3\setminus L_2, \ast) \cong \FFa\big({\{1,2\},\{1,2\}}\big)$.
Then the automorphism of the \aaagg\
induced by a homeomorphism of~$(D^3, L_2)$, relative to the boundary,
is prescribed by giving:
\begin{enumerate}[label=(\roman*)]
\item the effect on the $x_i$ generators of~{$\pi_1(D^3\setminus L_2, \ast)$};
\item  the effect on the $K_i$ generators of {$\pi_2(D^3\setminus L_2, \ast)$} as a ${\Z[\pi_1]}$-module.
\end{enumerate}
In particular we have the following.

\begin{Lemma} \label{L:Two_circles}
Recall the elementary braid permutation
$\Sigma_1\in \MCG(D^3,L_2)$
illustrated in Figure~\ref{F:loopin0} for~$n=2$.
  \redx{The image $\ssigma$ of $\Sigma_1\in \MCG(D^3,L_2)$
 in 
 $\Aut(\pi_{(1,2)}(D^3\setminus L_2,\ast))$ is given by:}
\[
\ssigma_{} (x_2) = x_2^{-1} x_1 x_2, \quad {\ssigma_{} (x_1) = x_2,}
\]
\begin{equation} \label{eq:ssigmaK2}
\ssigma_{} (K_2) = x_2^{-1} \trr K_1
\end{equation}
The image
of $K_1$ is determined
by these since  $\ssigma_{}$ fixes~$K_1 + K_2$.
\end{Lemma}

\begin{Remark} This result, and proof, is also used in~\cite[\S 4.5.1]{loopy}, where a dual form of the lifted Artin representation was treated, in the context of biracks derived from \agg s.
\end{Remark}
\proof {(Sketch)}
A  proof can be deduced from the observation of Figures~\ref{fig:OO1x} and~\ref{F:flips}.
It only remains to verify the last claim.
Consider the representatives of~$K_1 + K_2$ sketched in Figure~\ref{fig:K1K2}. The last one is evidently not moved by~{$\Sigma_1$}.
We will go back to these ideas later in Section~\ref{S:topological_interpretation}.
\endproof

\begin{figure}[!ht]
\centering
\includegraphics[scale=0.7]{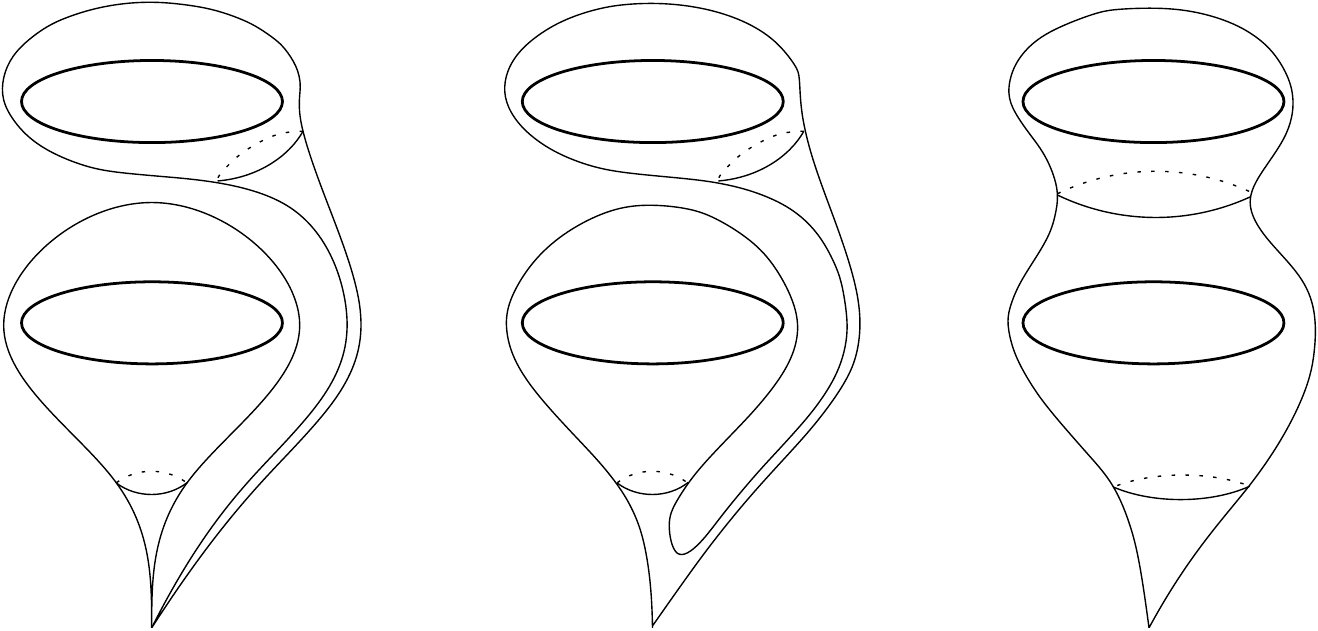}
\caption{Representatives of the element~$K_1 + K_2$.}
\label{fig:K1K2}
\end{figure}

Let us now treat this construction from a more algebraic viewpoint.
\begin{Definition} \label{de:Mmoncat}
  Define $\M$ as the monoidal category with object class $\N_0$ and
  $\M(n,n)= \aagg(\M_n, \M_n)$ with category composition as in~$\aagg$,
  and $\M(n,m)$ empty otherwise;
  and monoidal composition given by $a\otimes b = a+b$ on
  objects, and on morphisms as induced by the map
  $\{ 1, \ldots , n \} \sqcup \{ 1, \ldots m \} \mapsto
  \{ 1 , \ldots , n, {1+n}, \ldots , {m+n} \}$
  applied ``diagonally'' to the indices on pairs~$x_i, K_i$.
  Note that the permutation $(12)$ on $\{ 1,2 \}$ extends to make this
  a symmetric monoidal category.
\end{Definition}

\begin{Definition}
Recall Proposition \ref{extfromgens}.
Define $\mathcal{S}, \mathcal{R} \in \aagg(\M_2, \M_2)$ and  $\mathcal{T} \in \aagg(\M_1, \M_1)$
to be the uniquely defined morphisms of \agg s such that:
\begin{subequations}
\begin{align} 
\label{defofS1}
 \mathcal{S}^1_{}(x_j)&=
 \begin{cases} 
	x_{2}, \textrm{ if } j=1,\\
    x_{2}^{-1} x_1 x_{2}, \textrm{ if } j=2,\\
\end{cases}\\
\label{defofS2}
 \mathcal{S}^2_{}(K_j)&=
 \begin{cases} 
	K_{1}+K_{2}-x_{2}^{-1} \trr K_1 , \textrm{ if } j=1,\\
	x_{2}^{-1} \trr K_1 , \textrm{ if } j=2,\\
 \end{cases}
\end{align}
\end{subequations}

\begin{subequations}
\begin{align} 
\label{defofR1}
 \mathcal{R}^1_{}(x_j)&=
 \begin{cases} 
    x_{2}, \textrm{ if } j=1,\\
    x_1 , \textrm{ if } j=2,\\
 \end{cases}\\
\label{defofR2}
 \mathcal{R}^2_{}(K_j)&=
 \begin{cases} 
    K_{2} , \textrm{ if } j=1,\\
    K_1 , \textrm{ if } j=2,\\
 \end{cases} 
\end{align}
\end{subequations}

\begin{subequations}
\begin{align} 
\label{defofT1}
 \mathcal{T}^1_{}(x_1)&=   x_{1}^{-1},
 \hspace{1cm}
 \mathcal{T}^2_{}(K_j) = K_j.
\end{align}
\end{subequations}

Define $\mathcal{S}_i ,
 \mathcal{T}_i 
 \in \aagg(\M_n , \M_n)$, for each $i=1,\dots,n-1$, and $\mathcal{R}_i\in \aagg(\M_n , \M_n)$, for each $i=1,\dots,n$, as:
\[
\mathcal{S}_i = 1^{i-1} \otimes \mathcal{S} \otimes 1^{n-i-1} ,
\hspace{1.8cm}
\mathcal{R}_i = 1^{i-1} \otimes \mathcal{R} \otimes 1^{n-i-1} ,
\hspace{1.8cm}
\mathcal{T}_i = 1^{i-1} \otimes \mathcal{T} \otimes 1^{n-i} .
\]
\end{Definition}

Explicitly:
\begin{subequations}
\begin{align} 
 \mathcal{S}^1_i(x_j)&=\begin{cases} x_j, \textrm{ if } j<i \textrm{ and } j>i+1,\\
                          x_{i+1}, \textrm{ if } j=i,\\
                          x_{i+1}^{-1} x_i x_{i+1}, \textrm{ if } j=i+1,\\
           \end{cases}\\
 \mathcal{S}^2_i(K_j)&=\begin{cases} K_j, \textrm{ if } j<i \textrm{ and } j>i+1,\\
                          K_{i}+K_{i+1}-x_{i+1}^{-1} \trr K_i , \textrm{ if } j=i,\\
                          x_{i+1}^{-1} \trr K_i , \textrm{ if } j=i+1;\\
           \end{cases}
\end{align}
\end{subequations}
\begin{subequations}
\begin{align} 
 \mathcal{R}^1_i(x_j)&=\begin{cases} x_j, \textrm{ if } j<i \textrm{ and } j>i+1,\\
                          x_{i+1}, \textrm{ if } j=i,\\
                           x_i , \textrm{ if } j=i+1,\\
           \end{cases}\\
 \mathcal{R}^2_i(K_j)&=\begin{cases} K_j, \textrm{ if } j<i \textrm{ and } j>i+1,\\
                         K_{i+1} , \textrm{ if } j=i,\\
                          K_i , \textrm{ if } j=i+1;\\
           \end{cases} 
\end{align}
\end{subequations}
\begin{subequations}
\begin{align}
 \mathcal{T}^1_i(x_j)&=\begin{cases} x_j, \textrm{ if } j\neq i, \\
                          x_{i}^{-1}, \textrm{ if } j=i,
           \end{cases}\qquad \textrm{ and }
&& \mathcal{T}^2_i(K_j)= K_j.\end{align}
\end{subequations}

Remark that $\mathcal{T}_i^2: M_n \rightarrow M_n$ is not the identity. For
example $\mathcal{T}_i^2(x_i \trr K_j)= x_i^{-1} \trr K_j$.

\begin{Lemma}
These 
homomorphisms of \agg s
 are invertible.
\end{Lemma}
\begin{proof}
The claim is true for  
  $\mathcal{R}_i$ and the~$\mathcal{T}_j$, because 
$\mathcal{R}_i^2=\id$ and~$\mathcal{T}_j^2=\id$. 
Explicit calculation shows that the inverse of 
$\mathcal{S}_i$ is given by $\overline{\mathcal{S}_i}=(\overline{\mathcal{S}_i}^1,\overline{\mathcal{S}_i}^2)$ below:
\begin{subequations}
\begin{align} 
\label{defofS1m}
 \overline{\mathcal{S}}^1_i(x_j)&=
 \begin{cases} 
    x_i x_{i+1} x_i^{-1}, \textrm{ if } j=i,\\
    x_i,  \textrm{ if } j=i+1,\\
	x_j, \textrm{ if } j \neq i, i+1,\\
 \end{cases}\\
\label{defofS2m}
 \overline{\mathcal{S}}^2_i(K_j)&=
 \begin{cases} 
	x_{i}  \trr K_{i+1}  , \textrm{ if } j=i,\\
	K_i+K_{i+1}-  x_{i}  \trr K_{i+1} , \textrm{ if } j=i+1,\\
	K_j, \textrm{ if } j \neq i, i+1.\\
 \end{cases}
\end{align}
\end{subequations}
For the geometrical idea behind these calculations, see Lemma~\ref{L:Two_circles}.
\end{proof}

The following is the first main result of this paper.

\begin{Theorem}
\label{lfa}
\label{main1}
There exists a unique group homomorphism $\Pn\colon \LBE\nn \to \Aut(\M_n)$, such that:
\[
\Sigma_i \mapsto \mathcal{S}_i,  \quad \Rho_i \mapsto \mathcal{R}_i,  \quad \Rho_{{j}}\mapsto \mathcal{T}_j,
\]
where~$i=1,\dots, n-1$,  $j=1,\dots, n$, and~$\Sigma_i$, $\Rho_i$ and $\Tau_j$ are the generators of $\LBE\nn$ as a mapping class group, as in Figure~\ref{F:flips} . 
Furthermore, $\Pn\colon \LBE\nn \to \Aut(\M_n)$ is injective.
\end{Theorem}

  To $\Pn\colon  \LBE\nn \to \Aut(\M_n)$, we call the \textit{\TAR}.
\begin{proof}
  These images are sufficient to determine a homomorphism by Theorem~\ref{gensandrelsforLBG}, and the comments just after.
In Appendix~\ref{A:calculations} we check relations by explicit calculations.

 Let $\G=(G,A,\trr)$ be {a} \aaagg. Then we have a morphism~$F^1\colon \Aut(\G) \to 
 \Aut(G)$, such that~$F^1(f^1,f^2)=f^1$. 
 The injectivity of $\Pn$ follows from the fact that $F^1\circ \Pn$ coincides 
 with Dahm's homomorphism~$\LBE\nn \to \Aut(F_n)$, see 
 Theorem~\ref{gold-theorem}.
 \end{proof} 

A feature of Artin's representation,
is that it provides a characterisation of \emph{braid automorphisms}
of~$F_\nn$,
as  recalled in Theorem~\ref{artin_theorem}.
Also in the case of extended loop braid groups, Goldsmith gives a characterisation elements of~$\Aut(F_\nn)$ that are images of elements of~$\LBE\nn$, Theorem~\ref{gold-theorem}.
It is thus natural to ask the following question. 

\begin{OpenP} 
Determine the image of $\Pn\colon \LBE\nn \to \Aut(\M_n)$. 
\end{OpenP}

\begin{Remark}
\label{conservation_flux} 
Note that for any $(f^1,f^2) \in {\Pn(\LBE\nn)}$ it holds that:
\begin{equation}
\label{conservation_flux_eq}
f^2(K_1+\dots + K_n)=K_1+\dots + K_n.
\end{equation}
This is an analogue for condition~\eqref{prod-of-gens} in 
Theorem~\ref{artin_theorem}. As we will see in Remark~\ref{conservation_flux_reasons},  Equation~\eqref{conservation_flux_eq}
holds since $K_1+K_2+\dots+K_n$ can be represented by a parametrisation~$(D^2,\d D^2) \to (\d D^3,\ast)$, and hence is left untouched by all elements of~$\MCG(D^3,\C)$, given that these mapping classes are the identity in~$\d D^2$.
\end{Remark}

\begin{Remark}{A dual version of the \TAR, which considers finite dimensional representations of the loop braid group defined from biracks derived from \agg s, appeared in~\cite{loopy}. The underlying invariants of welded knots were initially defined in~\cite{kauffman_faria_martins_2008}.}
\end{Remark}

\section{A topological interpretation of the \TAR}
\label{S:topological_interpretation}

Let $X$ be a path-connected space.
A path $\gamma\colon [0,1] \to X$ canonically induces isomorphisms~$\pi_2(X,\gamma(1)) \to \pi_2(X,\gamma(0))$, 
and also $\pi_1(X,\gamma(1)) \to \pi_1(X,\gamma(0))$; see e.g.~\cite[Page 341]{Hatcher}. 
Let $\ast \in X$ be a base-point. 
If we consider closed paths, starting and ending in $\ast$, 
this 
descends to an action of $\pi_1(X,\ast)$ on~$\pi_2(X,\ast)$, 
by automorphisms; see e.g.~\cite[page 342]{Hatcher}. We denote by $\pi_{(1,2)}(X,\ast)$ the
\aaagg\ obtained from~$\pi_1(X,\ast)$ 
acting on 
$\pi_2(X,\ast)$ in this way.
We call
$ \pi_{(1,2)}(X,\ast)$  the
\emph{fundamental \agg\ of}~$(X,\ast)$ (confer with~\cite[Chapter I, \S 1]{Baues4D}).
We let $\TOP_\ast$ be the
category of pointed topological spaces, and pointed maps.
Define also $\TOP_\ast /_{\cong}$  as the
category with object the pointed spaces, and morphisms $(X,\ast) \to (Y,\ast)$ to be pointed maps $(X,\ast) \to (Y,\ast)$, considered up to 
pointed homotopy.
The fundamental \agg\ extends to functors $\pi_{(1,2)}\colon \TOP_\ast \to \aagg$ and~$\pi_{(1,2)}\colon \TOP_\ast /_{\cong} \to \aagg$.
Given a pointed map $f\colon (X,\ast) \to (Y,\ast)$, the induced map of  \agg s  
is denoted by~$\pi_{(1,2)}(f)\colon 
\pi_{(1,2)}(X,\ast) \to \pi_{(1,2)}(Y,\ast)$.
This is independent
of the representative
of the pointed homotopy class of~$f$.
Given a  pointed homotopy class of  pointed maps~$[f]\colon (X,\ast) \to (Y,\ast)$, 
the induced morphism
in $\aagg$
is also denoted by~$\pi_{(1,2)}([f])\colon 
\pi_{(1,2)}(X,\ast) \to \pi_{(1,2)}(Y,\ast)$. Note that~$\pi_{(1,2)}([f]) =\pi_{(1,2)}(f)$.

\subsection{Some bouquets and their fundamental \texorpdfstring{\agg}{agg}s}

\begin{Definition}
  Fix, from now on, base points $\star \in S^1$ and $\star \in S^2$, hence $(S^1,\star)$ and $(S^2,\star)$ will be well-pointed. Let $n$ be a positive integer.
For $i=1,\dots, n$, let $S^1_i$ be a  
copy of the circle~$S^1\subset \R^2$, oriented counterclockwise, and 
$S^2_i$ a  
copy of the {2-sphere}~$S^2\subset \R^3$, oriented by the exterior normal.
We define the following pointed spaces,
with the obvious CW-decompositions, where $\star$ is the unique $0$-cell:
\[
(\H_n,{\star})=\bigvee_{i=1}^n \big((S^1_i,{\star})\vee (S^2_i,{\star})),
\quad\quad  (\F_n,{\star})= \bigvee_{i=1}^n  (S^1_i,{\star}),
\quad\quad  (\J_n,{\star})=\bigvee_{i=1}^n  (S^2_i,{\star}).
\]
\end{Definition}

Let $A_i\in\pi_2(\H_n,\star)$
be  given by the homotopy class of a positively oriented
characteristic map~$(D^2,\d D^2) \to (S^2_i,\star)\subset
(\H_n,\star)$, of the 2-cell $S^1_i$ of~$\H_n$,  for~$i=1,\dots,n$.
Let $x_i\in\pi_1(\H_n,\star)$
be  given by the homotopy class of a  positively oriented characteristic map $([0,1],\{0,1\}) \to (S^1_i,\star)\subset (\H_n,\star)$  of the 1-cell $S^1_i$ of~$\H_n$, for~$i=1,\dots,n$.

By the Seifert -- van Kampen theorem, the group $\pi_1(\H_n,\star)$
is free on the~$x_1,\dots,x_n \in\pi_1(\H_n,\star)$.
Hence  $\pi_1(\H_n,\star)\cong F_n$ in \eqref{eq:deFn}
canonically.
This is because, since spaces are well-pointed, we have canonical isomorphisms of groups (see e.g.~\cite[Theorem 1.20
and Example 1.21]{Hatcher}):
\begin{align*}
\pi_1(\H_n,\ast)&= \pi_1\Big ( \bigvee_{i=1}^n \big((S^1_i,{\star})\vee (S^2_i,{\star})\big) \Big) \cong 
\bigvee_{i=1}^n \pi_1\Big (  \big((S^1_i,{\star})\vee (S^2_i,{\star})\big) \Big)\\
&\cong \bigvee_{i=1}^n \big(\pi_1\big(S^1_i,{\star})\vee \pi_1(S^2_i,{\star})\big)\cong  \bigvee_{i=1}^n \big(\pi_1\big(S^1_i,{\star})\big)\cong \bigvee_{i=1}^n\F^a(\{x_i\})\cong F_n.
\end{align*}
In particular, the inclusion  $(\F_n,\star)\to (\H_n,\star)$ canonically induces an isomorphism of groups $F_n\cong \pi_1(\F_n,\star) \stackrel{\cong}{\to}  \pi_1(\H_n,\star)\cong F_n$.

We now determine $\pi_2(\H_n,\star)$ as a $\pi_1(\H_n,\star)$-module.
\begin{Lemma}
\label{ultimate_freeness} 
As an abelian group, $\pi_2(\H_n,\ast)$ is freely generated
by the~$g\trr A_i$,
where $g \in \pi_1(\H_n,\star)\cong \pi_1(\F_n,\star) \cong F_n$, and~$i=1\dots,n$. We have an isomorphism sending $g \trr A_i$ 
to~$(g,K_i)$:
\[
\pi_2(\H_n,\ast)\to M_n, 
\]
where we recall from Definition~\ref{D:M_n} that: 
\[
M_n=\Z[F_n]\{K_1,\dots,K_n\}\cong\bigoplus_{i=1}^n 
\Z[F_n]{K_i}\cong \bigoplus_{g \in F_n}\bigoplus_{i=1}^n \Z(g,K_i).
\]
In particular, we have a  canonical isomorphism of \agg s: 
\[
\pi_{(1,2)}(\H_n,\ast) \cong \M_n.
\]
\end{Lemma}
 
\begin{proof}
  We extend~\cite[Example 4.27]{Hatcher}, which deals with 
case~$n=1$. The crucial ideas of the following argument are also in~\cite[\S 4.5.1]{loopy}.

Let $q\colon (\widehat{\H_n},\hat{\star}) \to (\H_n,\star)$ be the universal cover
of~$\H_n$, where we fixed a~$\hat{{\star}} \in q^{-1}({\star})$. Lifting elements of $\pi_2(\H_n,{\star})$ to 
elements of $\pi_2(\widehat{\H_n},\hat{{\star}})$ yields an isomorphism: 
\[
\Psi  \colon \pi_2(\H_n,{\star}) \to \pi_2(\widehat{\H_n},\hat{\star}) ,
\]
see e.g.~\cite[Proposition~4.1]{Hatcher}. 
Let ${\overline{q}\colon (\widehat{\F_n},\hat{\star})}\to {(\F_n,\star)}$ be the 
universal cover of~$\F_n$, where we fixed~$\hat{{\star}} \in \overline{q}^{-1}({\star})$. The crucial observation, as in~\cite[Example 4.27]{Hatcher}, is that that $\widehat{\H_n}$ is obtained 
from~$\widehat{\F_n}$  by attaching a copy of
$(\J_n,{\star})=\bigvee_{i=1}^n  (S^2_i,{\star})$, along~$\star$, to each element of~$\overline{q}^{-1}({\star})$.

Covering space theory gives a one-to-one correspondence~$\pi_1(\H_n,\star)\to q^{-1}({\star})=\overline{q}^{-1}({\star})$. We choose the 
convention sending $g \in \pi_1(\H_n,\star)$ to~$\hat{{\star}}\triangleleft g$, where on the 
right-hand-side we considered the monodromy right-action of $\pi_1(\H_n,{\star})\cong\pi_1(\F_n,{\star})$ 
on~$q^{-1}({\star})=\overline{q}^{-1}({\star})$.

For each~$g \in \pi_1(\H_n,\star)$, consider a copy of $(\J_{n,g},\star_g)$ of~$(\J_n,\star)$.
We have a pushout diagram:
\[ \xymatrix{
&\displaystyle{\bigsqcup_{g \in \pi_1(\H_n,\star)}} \ar[d]\{\hat{{\star}}\triangleleft g\}\ar[rr] && \ar[d] \displaystyle{ \bigsqcup_{g \in \pi_1(\H_{n},\star)} \J_{n,g}}\\
& \widehat{\F_n}\ar[rr]^{\iota_n}  && \widehat{\H_n}}
\]
in the category of topological spaces. Here the top horizontal arrow
is induced by~$\hat{{\star}}\triangleleft g\mapsto\star_g$.  In
particular, the top horizontal arrow is a cofibration. Consequently
$\iota_n\colon \widehat{\F_n} \to \widehat{\H_n}$ is a cofibration;
see e.g.~\cite[Page 44]{May}. Since $ \widehat{\F_n}$ is contractible,
we hence have a homotopy equivalence~$\widehat{\H_n} \to
\widehat{\H_n}/\widehat{\F_n}$, obtained by
collapsing
$\widehat{\F_n}$ to a point: call it $\star$; see e.g.~\cite[Proposition 0.17]{Hatcher}. 
Now note that  $(\widehat{\H_n}/\widehat{\F_n},\star)$ is canonically homeomorphic to~$\bigvee_{g \in  \pi_1(\H_n,\star)} 
\bigvee_{i=1}^n (S^2_{g,i},\star)$, where~$S^2_{g,i}=S^2$. 
Therefore we have a canonical isomorphism~$\pi_2(\widehat{\H_n},\star)\cong \pi_2\big(\bigvee_{g \in \pi_1(\H_n,\star)} 
\bigvee_{i=1}^n (S^2_{g,i},\star)\big )$.

By~\cite[Example 4.26]{Hatcher}, $\pi_2\big(\bigvee_{g \in \pi_1(\H_n,\star)} 
\bigvee_{i=1}^n (S^2_{g,i},\star)\big )$
 is the free abelian group on the homotopy classes $A_{g,i}$ yielded by the inclusions $(S^2,\star)\stackrel{\cong}{\to }  (S^2_{g,i},\star)\to \bigvee_{g \in  \pi_1(\H_n,\star)
} 
\bigvee_{i=1}^n (S^2_{g,i},\star)$. Therefore we have:
\[
\pi_2(\widehat{\H}_n,\hat{{\star}}) \cong \bigoplus_{g \in{\pi_1(\H_n,\star)}}\,\,\, \bigoplus_{i=1}^n 
\Z(g,K_i)
\]
canonically, where~$\Z(g,K_i) \cong \Z$, and its positive generator is identified with 
\[
A_{g,i}\in \pi_2\big(\bigvee_{g \in \pi_1(\H_n,\star)} 
\bigvee_{i=1}^n (S^2_{g,i},\star)\big )\cong  \pi_2(\widehat{\H_n},\star).
\]
With our conventions, the positive generator of $\Z(g,K_i)$ corresponds to~$\Psi 
(g \trr A_i)\in \pi_2(\widehat{\H_n},\star)$, for each $g \in \pi_1(\H_n,\star)$ and each~$i\in\{1,\dots,n\}$. Given that 
$\Psi\colon \pi_2(\H_n,{\star}) \to \pi_2({\widehat{\H_n}},\hat{{\star}})$ is an 
isomorphism, and that~$\pi_1(\H_n,\star)\cong F_n$, this completes the proof.
\end{proof}

\subsection{Fundamental \texorpdfstring{\agg}{agg} for \texorpdfstring{$(D^3 \setminus L_n)$}{D minus L}}
Recall that $\C=\cup_{i=1}^n C_i$ is an unlinked union of unlinked circles in the interior of~$D^3$. It is now convenient to consider that all circles are confined to the vertical plane~$\{y=1/2\}$. We take each $C_i$ to be a circle centred at~$(i/(n+1),1/2,1/2)$, with radius~$1/2n$, and oriented clockwise, for the point of view of an observer sitting in the $y=0$ plane. Furthermore, we consider a base-point located in the $\{z=0\}$ plane.
Analogously to the 2-dimensional case in Subsection~\ref{2dcase}, we have:
\begin{Proposition}
\label{defretP}
Let $\ast$ be a base point for $D^3 \setminus \C$ contained in~$\d D^3\cap\{z=0\}$. The 
spaces $(D^3 \setminus \C,\ast)$ and  $(\H_n,{\star})$ are pointed homotopic. 
Moreover, we have a deformation retraction from $(D^3 \setminus \C,\ast)$ onto a certain homeomorphic image of~$(\H_n,{\star})$. Here $(\H_n,{\star})$ embeds in $(D^3\setminus \C,\ast)$ in a way 
where each oriented circle $S^1_i$ encircles the oriented circle~$C_i\subset D^3$, forming a 
Hopf link, with positive linking number, and each oriented $S^2_i$ embeds as a positively oriented ``balloon'' (i.e. a 2-sphere, oriented by the exterior normal $\vec{n}_i$) containing~$C_i$, and no other circle, 
see Figure~\ref{def-ret}. 
\end{Proposition}
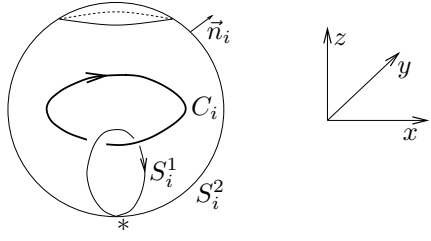
\begin{figure}[ht!]
 \labellist
\pinlabel ${S^2_i}$ at 315 35
\pinlabel ${S^1_i}$ at 243 70
\pinlabel ${C_i}$ at 307 171
\pinlabel ${\ast}$ at 180 -13
\pinlabel ${{\vec{n}_i}}$ at 329 290
\pinlabel ${{x}}$ at 626 130
\pinlabel ${{y}}$ at 617 230
\pinlabel ${{z}}$ at 517 279
\endlabellist
\centering
\includegraphics[scale=0.24]{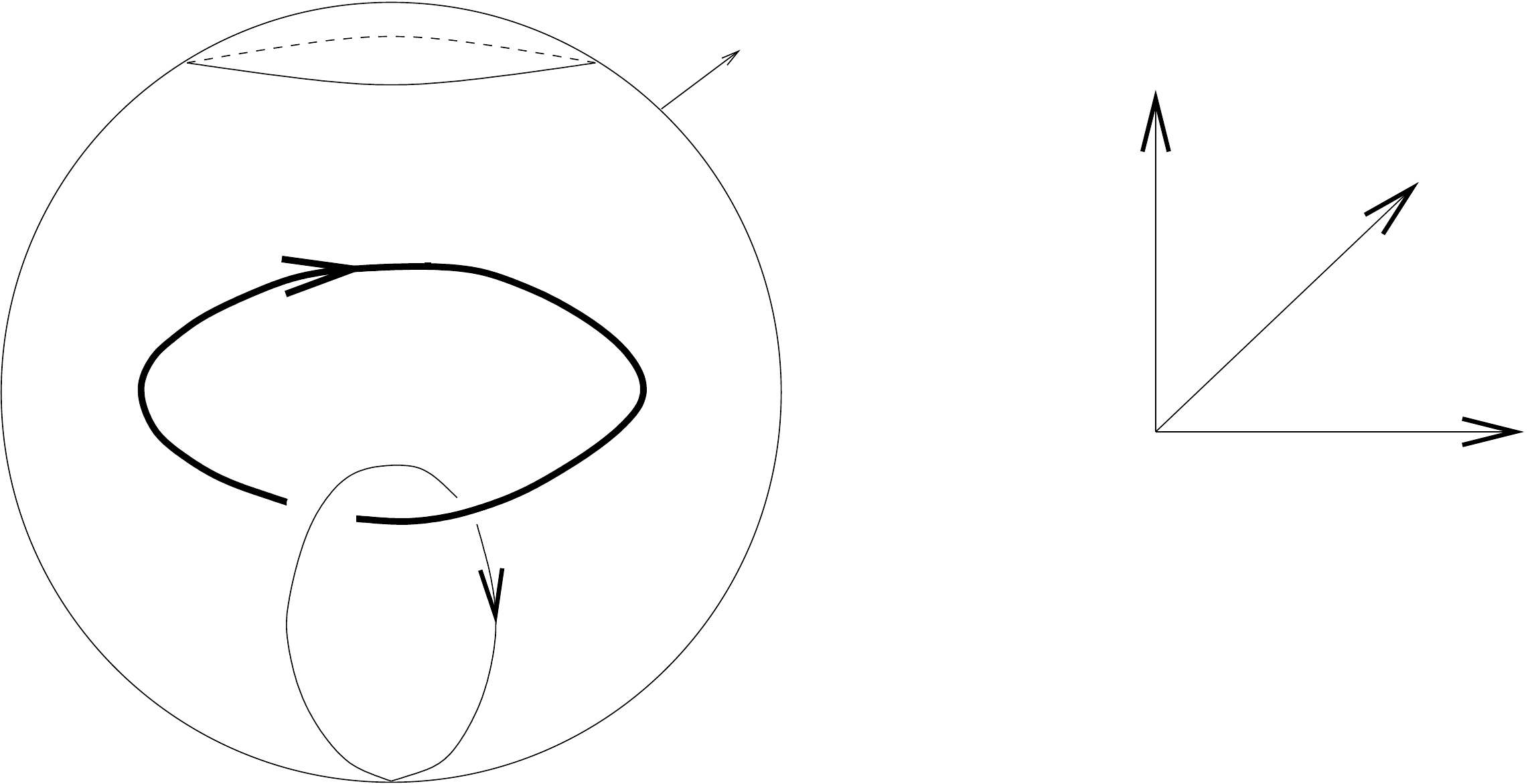}
\caption{Embedding of $(S^1_i,\star) \vee (S^2_i,\star)$ inside~$(D^3 \setminus 
\C,\ast)$.} 
\label{def-ret} 
\end{figure}

From the previous lemma, it follows that  
the inclusion $(\H_n,\star)$ inside  $\pi_2(D^3 \setminus \C,\ast)$ induces an isomorphism of \agg s~$\pi_{(1,2)}(\H_n,\star)  \cong  \pi_{(1,2)}(D^3\setminus \C,\ast)$. Hence $ \pi_{(1,2)}(D^3\setminus \C,\ast)\cong \M_n.$

Let $f$ be a self-homeomorphism of the pair~$(D^3, \C)$, relative to the boundary. 
Isotopies of such homeomorphisms are also considered to be relative to the boundary. 
Hence each element  $[f] \in \MCG(D^3,\C)=\LBE\nn$ yields, by restricting to~$D^3\setminus \C$, a pointed homotopy class of pointed maps $\K([f])\colon (D^3 \setminus \C,\ast) \to (D^3 \setminus \C,\ast)$; here recall that we imposed~$\ast\in \d D^3$. 
The induced map  of fundamental \agg s:
\[
\pi_{(1,2)}\big (\K([f])\big) \colon \pi_{(1,2)}(D^3 \setminus \C,\ast) \cong\M_n\longrightarrow \pi_{(1,2)}(D^3 \setminus \C,\ast)\cong \M_n
\]
is denoted $\Theta'([f])\colon \M_n \longrightarrow \M_n$. {The functoriality of $\pi_{(1,2)}\colon \TOP_\ast /_{\cong} \to \aagg$ gives that $\Theta'([f][g])=\Theta'([f\circ g])=\Theta'([f])\circ \Theta'([g])$, for each $[f],[g]\in \MCG(D^3,\C)$.}
We thus have a group homomorphism~$\Theta'\colon \LBE\nn =\MCG(D^3,\C) \to \Aut(\M_n)$. The following theorem gives the topological interpretation of the~{\TAR.}

\begin{Theorem}
\label{main2} 
The homomorphism $\Theta'\colon \LBE\nn\to \Aut(\M_n)$ coincides with the homomorphism $\Pn\colon \LBE\nn\to \Aut(\M_n)$ in Theorem~\ref{main1}. 
In particular $\Theta'$ is injective.
\end{Theorem}
The main ideas of the following proof are also in~\cite[\S 4.5.1]{loopy}.
\begin{proof}\textbf{{(Sketch)}}
We recall what each generator $g$ of the extended loop braid group in Theorem \ref{gensandrelsforLBG} incarnates geometrically. Let $\mathrm{Homeo}(D^3)$ be the group of homeomorphisms~$D^3\to D^3$, relative to the boundary. Consider an ambient isotopy $t \in [0,1]\mapsto \phi_t^g    {\in \mathrm{Homeo}(D^3)}$ of~$D^3$, 
for each generator $g$ of~$\LBE\nn$, as in Figure~\ref{corr}. Concretely $t \mapsto \phi_t^g$ is obtained by applying the isotopy extension theorem, as in~\cite[Chapter 8. 1.3 Theorem]{Hirsch}, to the smooth isotopies outlined in Figure~\ref{corr}.
Each {ambient} isotopy $t \mapsto \phi_t^g$ is relative to the boundary of~$D^3$, and satisfies~$\phi_0^g=\id_{D^3}$. Moreover, note that $\phi^g_1$ is a homeomorphism~$(D^3,\C) \to (D^3,\C)$.
The elements of $\MCG(D^3,\C)$ corresponding to each of the generators $g$ of $\LBE\nn$  in Theorem~\ref{gensandrelsforLBG} are obtained by evaluating the ambient isotopies  $\phi^g_t$  at~$t=1$. Our motions are the ones in~\cite[\S 3]{Damiani}, with the interval $[0, 1]$ taken with reversed extremes 
due to different conventions for the product in~$\MCG(D^3,\C)$.

\begin{figure}[hbt]
\centering
\includegraphics[scale=0.6]{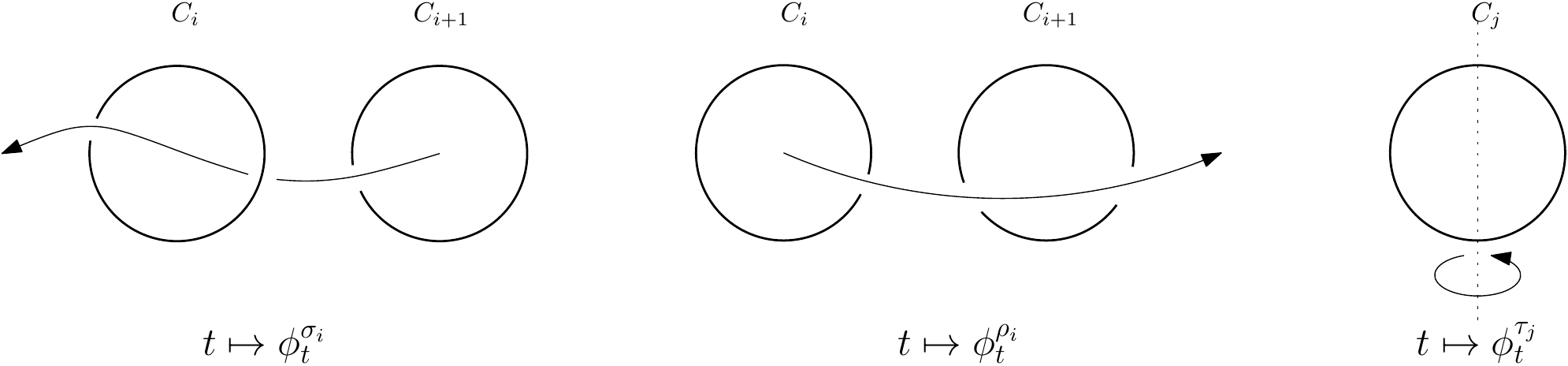}
\caption{Our conventions for the elements~$\Sigma_i$, $\Rho_i$ and $\Tau_j$ in~$\MCG(D^3,\C)$, obtained by evaluating (respectively) the shown isotopies at~$t=1$.}
\label{corr}
\end{figure}

Recall~$\pi_{(1,2)}(D^3,\C)\cong \M_n=(F_n,M_n,\trr)$,
where $M_n$ is the free $\Z[F_n]$-module generated by~$\{K_1, \ldots, K_\nn\}$; see Definition~\ref{D:M_n}.
Our conventions for the generators $x_j \in \pi_1(D^3 \setminus
\C,\ast)$ and~$K_j \in \pi_2(D^3 \setminus \C,\ast)$, where~$j=1,\dots,n$, are depicted on the left-hand side of
Figure~\ref{gens}. In particular, $K_j$ is obtained from a positively
  oriented parametrisation $(D^2,\d D^2) \to (D^3 \setminus
  \C,\ast_j)$ of a ``balloon'' (i.e. a $2$-sphere, oriented by an
  exterior normal), based in~$\ast_j$, as shown, containing the circle
  $C_j$ only; acted on by the path~$\gamma_j$, {a straight line}
  connecting $\ast$ and~$\ast_j$, in order to obtain an element of~$\pi_2(D^3\setminus \C,\ast)$. This point of view is as in~\cite{Ballons}, \cite[\S 4.5.1]{loopy} and~\cite[\S2.1.3]{martins_2009}.
Also, $x_j$ is given by the loop in Figure~\ref{gens}, which is then
conjugated by the path~$\gamma_j$, in order to yield an element of~$\pi_1(D^3\setminus \C,\ast)$.
By applying the deformation retration in Proposition~\ref{defretP}, we hence obtain the generators of $\pi_{(1,2)}(\H_n,\star)$ as in Lemma~\ref{ultimate_freeness}.
\begin{figure}[ht!]
 \labellist
 \pinlabel ${x_j}$ at 245 255
\pinlabel ${C_j}$ at 132 430
\pinlabel ${K_j}$ at 370 360
 \pinlabel ${x_{j+1}}$ at 772 255
\pinlabel ${C_{j+1}}$ at 665 430
\pinlabel ${K_{j+1}}$ at 909 360
\pinlabel ${\ast}$ at 342 -16
\pinlabel ${\ast}$ at 1530 -16
\pinlabel ${\ast_j}$ at 165 145
\pinlabel ${\ast_{j+1}}$ at 712 145
\pinlabel ${\gamma_j}$ at 286 83
\pinlabel ${\gamma_{j+1}}$ at 601 83
\pinlabel ${C_i}$ at 1295 430
\pinlabel ${C_{i+1}}$ at 1732 430
\endlabellist
\centering
\includegraphics[scale=0.2]{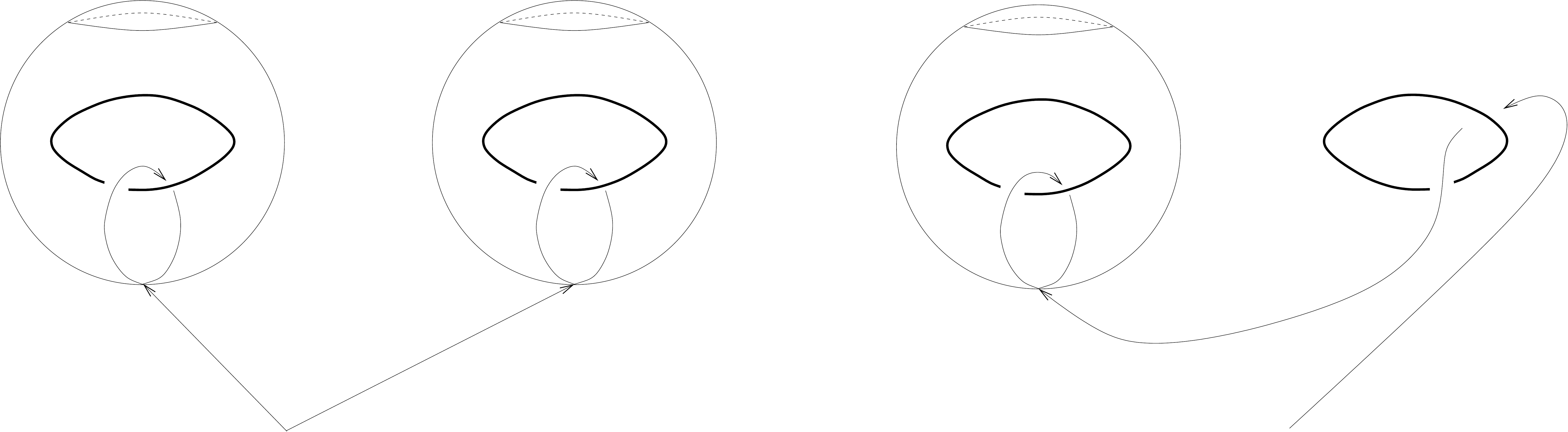}
\caption{\label{gens}
  On the left of the figure: our conventions for the generators~$x_j\in \pi_1(D^3 \setminus \C,\ast)$ and $K_j\in \pi_2(D^3
  \setminus \C,\ast)$, in the vicinity of the circle~$C_j$.   On the right of the figure:
  the result of acting with $\Sigma_i$ on $x_{i+1}$ and on~$K_{i+1}$.}
\end{figure}
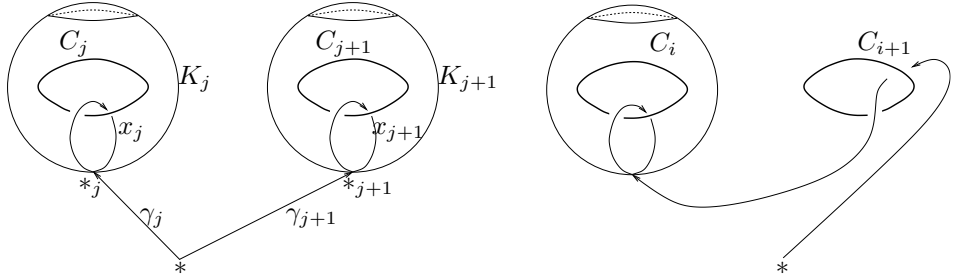

Let us see how the generators of $\MCG(D^3,\C)$ in Theorem \ref{gensandrelsforLBG} act on the $x_i$s and the~$K_j$s. 
We use the same notation to denote $g\in \LBE\nn$ and the map~$\Theta'(g)\colon \M_n \to \M_n$.

Whenever~$|i-j|>1$, nothing happens when we apply $\Sigma_i$ and $\Rho_i$ to $x_j$ and to~$K_j$. Also, $\Rho_i$ is such that:
\begin{align*}
\Rho_i \colon &&& x_i \mapsto x_{i+1}, & x_{i+1} \mapsto x_i, &  & K_i \mapsto K_{i+1}, && K_{i+1}\mapsto K_i.
\end{align*}
Hence $\Rho_i$ coincides with $\mathcal{R}_i$ in~\eqref{defofR1}
and~\eqref{defofR2}.

If we apply $\Sigma_i$ to $x_i$ we get~$x_{i+1}$.  The right-hand-side of Figure~\ref{gens} indicates what happens to $x_{i+1}$ and $K_{i+1}$ when we apply~$\Sigma_i$. Hence:
\begin{align*}\Sigma_i \colon &&& x_{i+1} \mapsto x_{i+1}^{-1}\, x_i \, x_{i+1}, & K_{i+1} \mapsto x_{i+1}^{-1} \trr {K_{i}. }
\end{align*} 

Let us now determine~$\Sigma_i(K_{i+1})$. Note that when the circle
$C_{i+1}$ goes inside the circle $C_{i}$ it ``drags'' the balloon
representing the class~$K_i\in \pi_2(D^3\setminus \C,\ast)$.
However~$\Sigma_i(K_{i}+K_{i+1} )=K_{i} + K_{i+1}$. This is because $K_i+K_{i+1}\in \pi_2(D^3\setminus \C,\ast)$ can be seen as being represented by a bigger balloon $K_{i,i+1}$ containing $C_i \cup C_{i+1}$, see Figure~\ref{balloon}. Moreover,  the isotopy $t\in [0,1] \mapsto \phi_t^{\sigma_i}$ from which we define $\Sigma_i=\phi_1^{\sigma_i}$ can be chosen to happen inside the bigger balloon~$K_{i,i+1}$, hence not moving $K_{i,i+1}$.
\begin{figure}[ht!]
 \labellist
\pinlabel ${C_i}$ at 166 313
\pinlabel ${C_{i+1}}$ at 463 313 
\pinlabel ${K_{i,i+1}}$ at 77 56
\pinlabel ${\ast}$ at 596 -12
\endlabellist
\centering
\includegraphics[scale=0.2]{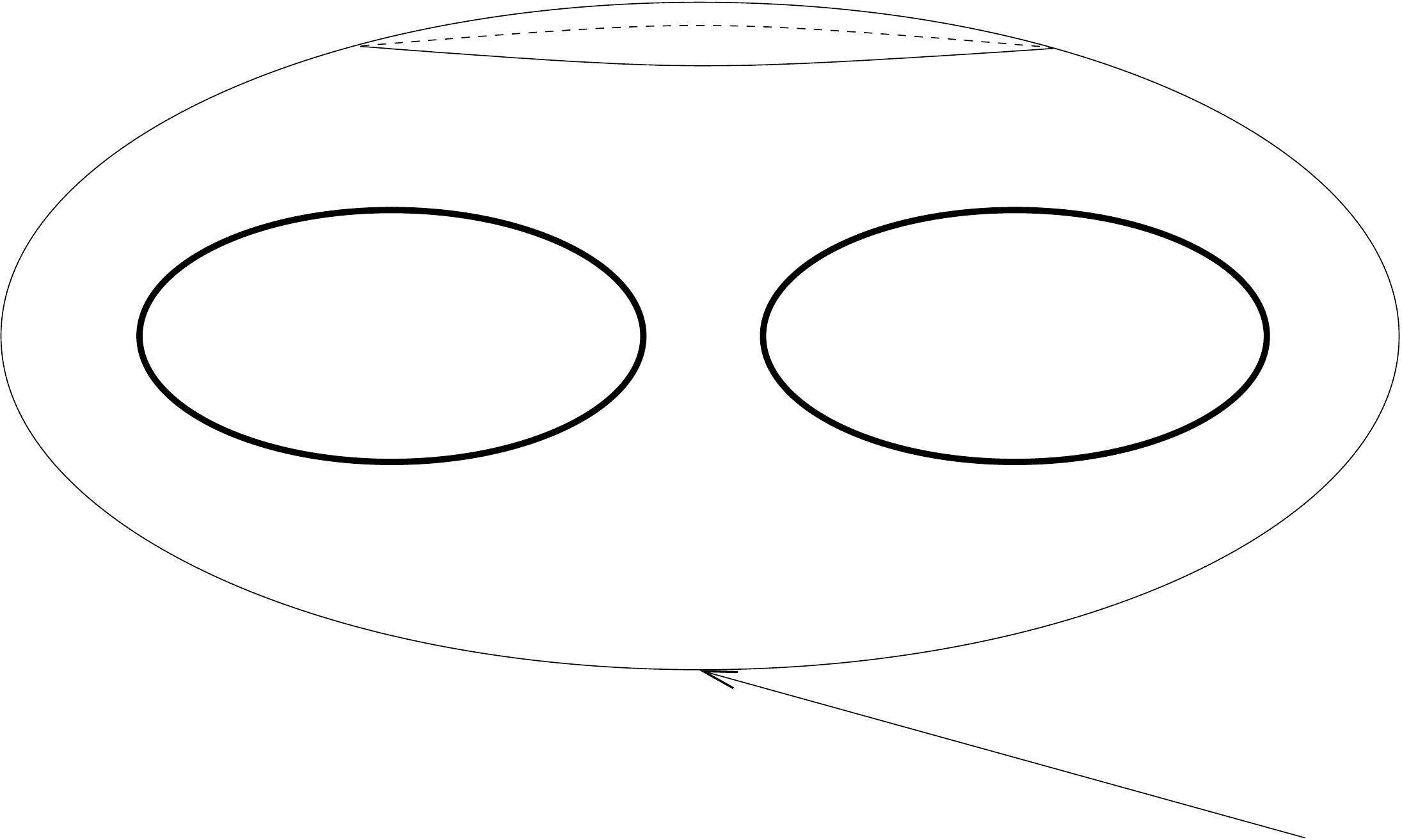}
\caption{\label{balloon} The balloon  $K_{i,i+1}$  representing $K_i+K_{i+1}\in \pi_2(D^3 \setminus \C,\ast)$.}
\end{figure}
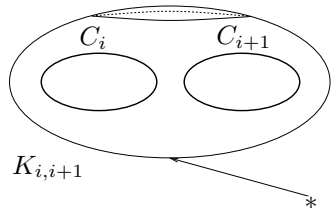

We obtain~$\Sigma_i(K_{i,i+1})=K_{i,i+1}$. Therefore: 
\[
\Sigma_i(K_i)+\Sigma_i(K_{i+1})=\Sigma_i(K_{i}+K_{i+1} )=K_{i} + K_{i+1}, 
\]
and
\[
\Sigma_i (K_i)=K_{i} + K_{i+1}-\Sigma_i(K_{i+1})=K_{i} + K_{i+1}-x_{i+1}^{-1}\trr K_{i+1}. 
\]
Hence $\Sigma_i$ coincides with $\mathcal{S}_i$ in~\eqref{defofS1} and~\eqref{defofS2}. 
That $\Tau_i$s coincide with $\mathcal{T}_i$s in~\eqref{defofT1} can be seen in a similar way. Note that~$\Tau_i(K_i)=K_i$, since the isotopy yielding $K_i$ can be chosen to happen inside the balloon representing~$K_i$.
\end{proof}

\begin{Remark}
  The restriction of the 
  representation
  $\Pn\colon \LB\nn^{\mathrm{ext} }\to \Aut(\M_n)$ to the loop braid
  group~$\LB\nn$, as well as its topological interpretation,
  are also discussed  
  in~\cite[\S 4.5.1]{loopy}. 
\end{Remark}

\begin{Remark}\label{conservation_flux_reasons}
In Remark~\ref{conservation_flux}, we promised an interpretation of the property~$\Pn(g)^2(K_1+\dots + K_n)=K_1+\dots + K_n$, for each~$g \in \MCG(D^3,\C)$.
Observe that, given our conventions, $M=K_1+\dots + K_n$  is homotopic to the element of~$\pi_2(D^3 \setminus \C,\ast)$ yielded by the inclusion of $(\d D^3,\ast)$ inside~$(D^3 \setminus \C,\ast)$. Since all elements of $\MCG(D^3,\C)$ restrict to the the identity over~$\d D^3$, it follows that~$\Pn(g)^2(M)=M$, for all~$g \in  \MCG(D^3,\C)=\LBE\nn$. 
\end{Remark}

\subsection{Remarks on connections to Baues' combinatorial homotopy}\label{remarks}
Given a pointed space $(X,\ast)$, we denote by $\E(X,\ast)$ the group of
pointed homotopy equivalences~$(X,\ast)\to (X,\ast)$, up to pointed
homotopy.
If~$\ast\in\d D^3$,
we have a group homomorphism~$\K\colon \LBE\nn \to \E(D^3 \setminus \C,\ast)$,
obtained by restricting $f \colon (D^3,  \C) \to (D^3, \C)$ to $D^3
\setminus \C$. 
The fundamental \agg\ functor ${\TOP_\ast}/_{\cong} \to \aagg$ gives another group homomorphism
$\E(D^3 \setminus \C,\ast)
\to \Aut(\pi_{(1,2)}(D^3 \setminus \C,\ast))\cong\Aut(\M_n)$. Composing with~$\K$,
this yields a group homomorphism~$\Theta'\colon \LB\nn \to \Aut(\M_n)$,
which by Theorem~\ref{main2}, coincides with~$\Pn\colon \LB\nn \to \Aut(\M_n) $.

Note that
taking into account our results,
it follows from elementary algebraic topological techniques,
or as a particular case of~\cite[ Theorem III(7.1)]{Baues4D} or \cite[ Corollary VI(3.5)]{BauesAH},
that  $\E(D^3 \setminus \C,\ast)\cong \E(\H_n,\star)\cong \Aut(\M_n)$.
This is a generalisation of the well known fact that~$\E(D^2 \setminus d_n,\ast) \cong \E(\F_n,*)\cong \Aut(F_n)$.
Hence in this $(D^3 \setminus \C,\ast)$ case, \agg s
retain all of the homotopy theoretical information necessary to combinatorially model the group~$\E(D^3 \setminus \C,\ast)$.
Therefore, the lifted Artin representation $\Theta\colon \MCG(D^3,\C) \to \Aut(\M_n)$ completely models the pointed homotopy type of the restriction of a mapping class in $\MCG(D^3,\C) $ to~$D^3 \setminus \C.$

\appendix
 
\section{Calculations}
\label{A:calculations}
In this Appendix we explicitly verify that relations from Theorem~\ref{main1} hold. It will be enough to consider the case~$\nn=3$. In order to reduce indices, with respect to Theorem~\ref{main1} we replace $\{x_1,x_2,x_3\}$ with $\{x,y,z\}$ and $\{K_1,K_2,K_3\}$ with $\{K,L,M\}$.
We will only show computation for generators $\{K,L,M\}$, since on $\{x,y,z\}$ our representation coincides with Dahm's homomorphism. 
\begin{itemize}
\item $\mathcal{S}_1 \circ \mathcal{S}_2 \circ \mathcal{S}_1=\mathcal{S}_2 \circ \mathcal{S}_1 \circ \mathcal{S}_2$.

Applying $\mathcal{S}_1 \circ \mathcal{S}_2 \circ \mathcal{S}_1$ one gets
\[
\left\{\begin{array}{llll}
&K \longmapsto &K+L-y^{-1} \trr K  \longmapsto K+ L+M-z^{-1} \trr L-z^{-1} \trr K  \longmapsto\\
&& K+L-y^{-1} \trr K + y^{-1} \trr K+M-z^{-1} \trr y^{-1} \trr K -z^{-1} \trr (K+L-y^{-1} \trr K) \\
&&  = K+L+M -z^{-1} \trr (K+L) , \\
&L   \longmapsto &  y^{-1} \trr K  \longmapsto z^{-1} \trr K \longmapsto  z^{-1}\trr (K+L-y^{-1} \trr K), \\
&M  \longmapsto & M \longmapsto z^{-1} \trr L \longmapsto z^{-1} \trr y^{-1} \trr K 
\end{array}\right.
\]
\noindent while applying $\mathcal{S}_2 \circ \mathcal{S}_1 \circ \mathcal{S}_2$ one gets
\[
\left\{\begin{array}{llll}
&K \longmapsto &K \longmapsto K+L-y^{-1} \trr K \longmapsto K+L+M-z^{-1} \trr L-z^{-1} \trr K, \\
&L   \longmapsto &L+M-z^{-1} \trr L \longmapsto  y^{-1} \trr K+M-z^{-1} \trr y^{-1} \trr K \\
&& \longmapsto z^{-1} \trr K+z^{-1} \trr L-(z^{-1}y^{-1}z) \trr z^{-1} \trr K \\
&M  \longmapsto & z^{-1} \trr L \mapsto z^{-1} \trr y^{-1}\trr K \mapsto (z^{-1}y^{-1}z) \trr z^{-1}\trr K.
\end{array} \right.
\]

\item $\mathcal{R}_2 \circ \mathcal{R}_1 \circ \mathcal{S}_2=\mathcal{S}_1 \circ \mathcal{R}_2 \circ \mathcal{R}_1$
Applying $\mathcal{R}_2 \circ \mathcal{R}_1 \circ \mathcal{S}_2$ one gets

\[
\left\{\begin{array}{llll}
&K \longmapsto & K \longmapsto L \longmapsto M \\
&L   \longmapsto &  L+M-z^{-1}\trr L \longmapsto K+M-z^{-1}\trr K \longmapsto K+L-y^{-1}\trr K \\
&M  \longmapsto & z^{-1} \trr L \mapsto z^{-1} \trr K \mapsto y^{-1} \trr K 
\end{array}\right.
\]
\noindent while applying $\mathcal{S}_1 \circ \mathcal{R}_2 \circ \mathcal{R}_1$ one gets
\[
\left\{\begin{array}{llll}
&K   \longmapsto &  L \longmapsto M \longmapsto M \\
&L   \longmapsto &  K \longmapsto K \longmapsto K+L-y^{-1} \trr K\\
&M   \longmapsto &  M \longmapsto L \longmapsto y^{-1} \trr K.
\end{array}\right.
\]

\item $\mathcal{S}_2\circ \mathcal{S}_1 \circ \mathcal{R}_2 =\mathcal{R}_1 \circ \mathcal{S}_2 \circ \mathcal{S}_1$ 

Applying $\mathcal{S}_2\circ \mathcal{S}_1 \circ \mathcal{R}_2 $ one gets

\[
\left\{\begin{array}{llll}
&K \longmapsto & K \longmapsto  K+L-y^{-1} \trr K \longmapsto K+L+M-z^{-1} \trr L-z^{-1} \trr K \\
&L   \longmapsto & M \longmapsto M \longmapsto z^{-1} \trr L \\
&M  \longmapsto &  L \longmapsto y^{-1} \trr K \longmapsto z^{-1}\trr K 
\end{array}\right.
\]
\noindent while applying $\mathcal{R}_1 \circ \mathcal{S}_2 \circ \mathcal{S}_1$ one gets
\[
\left\{\begin{array}{llll}
&K \longmapsto & K+L-y^{-1} \trr K \longmapsto K+L+M-z^{-1} \trr L-z^{-1} \trr K \\ 
&& \longmapsto L+K+M-z^{-1} \trr K-z^{-1} \trr L \\
&L   \longmapsto & y^{-1} \trr K \longmapsto z^{-1} \trr K \longmapsto z^{-1} \trr L \\
&M  \longmapsto &  M \longmapsto z^{-1} \trr L \longmapsto z^{-1}\trr K .
\end{array}\right.
\]

\item $\mathcal{T}_1 \circ \mathcal{R}_1=\mathcal{R}_1 \circ \mathcal{T}_2$
Applying $\mathcal{T}_1 \circ \mathcal{R}_1$ one gets 

\[
\left\{\begin{array}{llll}
&K \longmapsto & L \longmapsto L \\ 
&L   \longmapsto & K \longmapsto K \\
&M  \longmapsto &  M \longmapsto M
\end{array}\right.
\]

\noindent while applying $\mathcal{R}_1 \circ \mathcal{T}_2$ one gets

\[
\left\{\begin{array}{llll}
&K \longmapsto & K \longmapsto L \\ 
&L   \longmapsto & L \longmapsto K \\
&M  \longmapsto &  M \longmapsto M.
\end{array}\right.
\]

\item $\mathcal{T}_1 \circ \mathcal{S}_1 = \mathcal{S}_1 \circ \mathcal{T}_2$

Applying $\mathcal{T}_1 \circ \mathcal{S}_1$ one gets 

\[
\left\{\begin{array}{llll}
&K \longmapsto &K+L-y^{-1} \trr K \longmapsto K+L+ y^{-1} \trr K \\ 
&L   \longmapsto & y^{-1} \trr K \longmapsto y^{-1} \trr K \\
&M  \longmapsto &  M \longmapsto M
\end{array}\right.
\]

\noindent while applying $\mathcal{S}_1 \circ \mathcal{T}_2$ one gets

\[
\left\{\begin{array}{llll}
&K \longmapsto & K+L-y^{-1} \trr K \longmapsto K+L+ y^{-1} \trr K \\ 
&L   \longmapsto & y^{-1} \trr K \longmapsto y^{-1} \trr K \\
&M  \longmapsto &  M \longmapsto M.
\end{array}\right.
\]

\item $\mathcal{T}_2 \circ \mathcal{S}_1 = \mathcal{R}_1 \circ \mathcal{S}_1^{-1}\circ \mathcal{R}_1 \circ \mathcal{T}_1$

Applying $\mathcal{T}_2 \circ \mathcal{S}_1 $ one gets 

\[
\left\{\begin{array}{llll}
&K \longmapsto & K+L-y^{-1} \trr K \longmapsto K+L-y \trr K \\ 
&L   \longmapsto & y^{-1} \trr K \mapsto y \trr K\\
&M  \longmapsto &  M \longmapsto M
\end{array}\right.
\]

\noindent while applying $\mathcal{R}_1 \circ \mathcal{S}_1^{-1}\circ \mathcal{R}_1 \circ \mathcal{T}_1$ one gets

\[
\left\{\begin{array}{llll}
&K \longmapsto &   K \longmapsto L \longmapsto  K+L-x\trr L \longmapsto  L+K-y\trr K \\ 
&L   \longmapsto & L \longmapsto K \longmapsto x \trr L \longmapsto y \trr K \\
&M  \longmapsto &  M \longmapsto M \longmapsto M \longmapsto M.
\end{array}\right.
\]

\end{itemize}


\bibliographystyle{alpha}
\bibliography{2-sorted}

\end{document}